\documentclass[12pt]{article}

\usepackage{fullpage}
\usepackage{amssymb}
\usepackage{amsmath}
\usepackage{amsthm}
\usepackage{stmaryrd}
\usepackage{xypic}
\usepackage{setspace}
\usepackage{array}
\usepackage{rotating}

\newtheorem{theorem}{Theorem}[section]
\newtheorem{lemma}[theorem]{Lemma}
\newtheorem{proposition}[theorem]{Proposition}
\newtheorem{corollary}[theorem]{Corollary}

\newcommand{\tref}[1]{Theorem~\textup{\ref{thm:#1}}}
\newcommand{\pref}[1]{Proposition~\textup{\ref{prop:#1}}}
\newcommand{\cref}[1]{Corollary~\textup{\ref{cor:#1}}}
\newcommand{\eref}[1]{Equation~\textup{\ref{eq:#1}}}
\newcommand{\lref}[1]{Lemma~\textup{\ref{lem:#1}}}

\newcommand{\Z}{\mathbb Z}

\newcommand{\s}{\sigma}
\newcommand{\G}{\Gamma}
\newcommand{\mix}{\diamond}
\newcommand{\comix}{\boxempty}
\newcommand{\comment}[1]{}
\newcommand{\eps}{1}

\newcommand{\ol}[1]{\overline{#1}}
\newcommand{\GP}[2]{\G(#1 \mid #2)}

\newcommand{\calP}{\mathcal P}
\newcommand{\calQ}{\mathcal Q}

\newcommand{\calR}{\mathcal R}

\begin{document}

\title{Tight Chiral Polyhedra}

\author{Gabe Cunningham\\
Department of Mathematics\\
University of Massachusetts Boston\\
}

\date{ \today }
\maketitle

\begin{abstract}
	A chiral polyhedron with Schl\"afli symbol $\{p, q\}$ is called \emph{tight} if it has $2pq$ flags,
	which is the minimum possible. In this paper, we fully characterize the Schl\"afli symbols of tight chiral
	polyhedra. We also provide presentations for the automorphism groups of several families of
	tight chiral polyhedra.
	

\comment{
	20B25: Finite automorphism groups of algebraic, geometric, or combinatorial structures.
	51M20: Polyhedra and polytopes; regular figures, division of spaces
	52B05: Combinatorial properties of polyhedra (number of faces etc.)
	52B15: Symmetry properties of polytopes
}

\end{abstract}

\section{Introduction}

	If you take a convex polyhedron and only consider the ways in which
	the vertices, edges, and faces connect to each other (disregarding
	distances and angles), then you get the \emph{face-lattice}
	of that polyhedron. An \emph{abstract polyhedron} is essentially
	a partially-ordered set that resembles the face lattice of
	a convex polyhedron in certain key ways. Indeed, the face lattice of any convex
	polyhedron, plane tessellation, or face-to-face tiling of the torus is an abstract
	polyhedron.
	
	Central to the study of polyhedra (abstract or otherwise) 
	is the characterization of their symmetry.
	The symmetries of abstract polyhedra are order-preserving
	bijections, called \emph{automorphisms}; in other words, an automorphism is
	a way of shuffling the vertex, edge, and face labels without
	changing the incidence relationship. Automorphisms therefore also
	permute the \emph{flags} of the polyhedron, which consist
	of a vertex, an edge that is incident on that vertex, and a face 
	that is incident on that edge. The most symmetric polyhedra
	are \emph{regular}, where given any two flags, there is
	an automorphism that brings the first flag to the second.
	A polyhedron is \emph{chiral} if we can divide its flags
	into two classes such that flags that differ in a single
	element lie in different classes, and such that there is
	an automorphism bringing one flag to another if and only if
	those two flags lie in the same class.

	It follows as well that the faces of a regular or chiral polyhedron are all
	isomorphic, and the vertices all have the same valence.
	We say that a polyhedron has \emph{Schl\"afli symbol} $\{p, q\}$
	(or is \emph{of type $\{p, q\}$}) if the faces are all
	$p$-gons and the vertices are all $q$-valent. 

	What are the smallest chiral polyhedra? Any polyhedron
	of type $\{p, q\}$ has at least $2pq$ flags, and if
	the number of flags is exactly $2pq$ then the polyhedron is called
	\emph{tight} (see \cite{tight-polytopes}). In \cite{tight2} and
	\cite{tight3}, the authors 
	determined which Schl\"afli symbols occur among tight regular
	polyhedra and provided presentations for their automorphism
	groups.
	This paper builds on that work to determine the
	Schl\"afli symbols that occur among tight chiral polyhedra.

	In Section 2, we provide further background on abstract
	polyhedra, chirality, and tightness. Then in Section 3,
	we describe several infinite families of tight chiral
	polyhedra. Section 4 investigates the covering relations
	between tight chiral polyhedra and characterizes the tight
	chiral polyhedra that do not cover any other tight chiral
	polyhedra. Section 5 brings everything together to produce
	our main result, \tref{all-tight-chirals}, which fully characterizes the
	Schl\"afli symbols of tight chiral polyhedra. Then in Section 6, 
	we discuss future directions.
	
\section{Background}

	Our background on abstract polyhedra comes from \cite{arp} and \cite{chiral}.
	
	Let $\calP$ be a ranked partially-ordered set with elements of rank $0$, called \emph{vertices},
	elements of rank $1$, called \emph{edges}, and elements of rank $2$, called \emph{faces}.
	Let us say that two elements $F$ and $G$ are \emph{incident} if $F \leq G$ or $G \leq F$.
	A \emph{flag} of $\calP$ consists of a vertex, edge, and face that are all mutually incident.
	Then $\calP$ is an \emph{abstract polyhedron} if all of the following are true:
	
	\begin{enumerate}
	\item[(1)] Every edge is incident to exactly two vertices and two faces.
	\item[(2)] Whenever a vertex is incident to a face, there are exactly two edges that are incident to both.
	\item[(3)] The graph of the incidence relation is connected.
	\item[(4)] For any face or vertex $F$, the subgraph of the graph of the incidence relation induced by the
		neighbors of $F$ (not including $F$ itself) is connected.
	\end{enumerate}

	From now on, we will refer to abstract polyhedra simply as ``polyhedra''.
	
	Given any flag $\Phi$ of a polyhedron $\calP$ and $i \in \{0, 1, 2\}$,
	there is a unique flag $\Phi^i$ that differs from $\Phi$ only in its element of rank $i$.
	Two flags that differ in only a single element are said to be \emph{adjacent}.
	
	Whenever a face of a polyhedron is incident to $p$ edges, it must also be
	incident to $p$ vertices. These edges and vertices occur in a single cycle, and we say
	that the face is a \emph{$p$-gon}. Similarly, whenever a vertex is incident to $q$ edges, then it
	is also incident to $q$ faces, occurring in a single cycle. In this case we say that
	the \emph{vertex-figure} is a $q$-gon. If $\calP$ is a polyhedron whose
	faces are all $p$-gons and whose vertex-figures are all $q$-gons, then
	we say that $\calP$ has \emph{Schl\"afli symbol $\{p, q\}$}, or that it is of \emph{type $\{p, q\}$}.
	
	The \emph{dual} of a polyhedron $\calP$
	is the polyhedron obtained by reversing the partial order of $\calP$. 
	If $\calP$ has Schl\"afli symbol $\{p, q\}$, then its dual has Schl\"afli symbol $\{q, p\}$.

	\subsection{Regular and chiral polyhedra}
	
		An \emph{automorphism} of a polyhedron $\calP$ is an incidence-preserving bijection
		from $\calP$ to itself. The group of automorphisms of $\calP$ is denoted by $\G(\calP)$.
		We say that $\calP$ is \emph{regular} if $\G(\calP)$ acts transitively on the flags of $\calP$,
		We say that $\calP$ is \emph{chiral} if there are two classes of flags such that
		$\G(\calP)$ acts transitively on each class but not on the whole set of flags, 
		and such that if $\Phi$ is in
		one class, then all flags adjacent to $\Phi$ lie in the other class.

		Let $\calP$ be a regular or chiral polyhedron, and let us fix a \emph{base flag}
		$\Phi$. The \emph{rotation group} of $\calP$, denoted $\G^+(\calP)$, is the
		group generated by automorphisms $\s_1$ and $\s_2$, where $\s_1$ is the unique automorphism
		that sends $\Phi$ to $(\Phi^1)^0$ and $\s_2$ is the unique automorphism
		that sends $\Phi$ to $(\Phi^2)^1$. If $\calP$ is regular, then $\G^+(\calP)$
		either coincides with $\G(\calP)$ or has index $2$ in $\G(\calP)$; if the latter
		is true then we say that $\calP$ is \emph{orientably regular}. 
		If $\calP$ is chiral, then $\G^+(\calP)$ coincides with the full automorphism
		group $\G(\calP)$. 
		
		If $\calP$ is a regular or chiral polyhedron of type $\{p, q\}$, then
		$\G^+(\calP)$ satisfies at least the relations $\s_1^p = \s_2^q = (\s_1 \s_2)^2 = \eps$.
		Conversely, given any group $\G = \langle \s_1, \s_2 \rangle$ that satisfies
		those relations, there is a standard way to build a poset $\calP$ such that
		$\G^+(\calP) = \G$, and $\calP$
		will be a polyhedron if $\langle \s_1 \rangle \cap \langle \s_2 \rangle$ is trivial
		(see \cite{chiral} for details on the construction).
		In this case, $\calP$ will be orientably regular if $\G$ has an automorphism
		that sends each $\s_i$ to $\s_i^{-1}$, and chiral otherwise.
		In any case, we define the \emph{enantiomorphic form of $\G$} to be the
		image of $\G$ under the isomorphism that sends each $\s_i$ to $\s_i^{-1}$.
		Equivalently, we get a presentation for the enantiomorphic form of $\G$ by taking a
		presentation for $\G$ and changing every $\s_i$ in the defining relations to $\s_i^{-1}$.
		If $\G$ is the rotation group of an orientably regular polyhedron, then 
		the original defining relations still hold.
		Otherwise, if $\G$ is the automorphism group of a chiral polyhedron $\calP$,
		then we obtain a presentation where the original defining relations do not hold. 
		This gives us the automorphism group of the \emph{enantiomorphic form of $\calP$},
		which has the same underlying polyhedron as $\calP$, but with a different choice
		of base flag.

		If $\G = \langle \s_1, \s_2 \rangle$ and $\Lambda = \langle \lambda_1, \lambda_2 \rangle$,
		we say that $\G$ \emph{covers} $\Lambda$ if there is a well-defined homomorphism
		sending $\s_1$ to $\lambda_1$ and $\s_2$ to $\lambda_2$. Such a homomorphism must
		be surjective, justifying our terminology. Indeed, such a homomorphism exists
		exactly when $\Lambda$ satisfies all of the defining relations
		of $\G$ (with each $\s_i$ changed to $\lambda_i$). If $\G^+(\calP)$ covers
		$\G^+(\calQ)$ for orientably regular or chiral polyhedra $\calP$ and
		$\calQ$, then we also say that \emph{$\calP$ covers $\calQ$}.
	
		A chiral polyhedron $\calP$ of type $\{p, q\}$ has a unique \emph{minimal regular cover} $\calR$,
		also of type $\{p, q\}$ \cite{mix-ch}. Any regular polyhedron $\calQ$ that covers
		$\calP$ must also cover $\calR$.
		
		Suppose $\calP$ is an orientably regular or chiral polyhedron with $\G^+(\calP) = \langle \s_1, \s_2 \rangle$,
		and let $\calP^{\delta}$ be the dual of $\calP$. Then $\G^+(\calP^{\delta})$ has distinguished
		generators $\ol{\s_1}$ and $\ol{\s_2}$, where $\ol{\s_1} = \s_2^{-1}$ and $\ol{\s_2} = \s_1^{-1}$. In other words, if
		we have a presentation for $\G^+(\calP)$, we can obtain a presentation for $\G^+(\calP^{\delta})$
		by changing every $\s_i$ to $\s_{3-i}^{-1}$ in every defining relation.

	\subsection{Mixing polyhedra}
	
		Information in this section is taken chiefly from \cite{mix-ch}.
		Given two chiral or orientably regular polyhedra $\calP_1$ and $\calP_2$, there is a natural
		way to find the minimal common cover of their rotation groups $\G^+(\calP_1)$ and
		$\G^+(\calP_2)$. If $\G^+(\calP_1) = \langle \s_1, \s_2 \rangle$ and $\G^+(\calP_2) = \langle
		\s_1', \s_2' \rangle$, then we define the \emph{mix} of $\G^+(\calP_1)$ and $\G^+(\calP_2)$,
		denoted $\G^+(\calP_1) \mix \G^+(\calP_2)$,
		to be the subgroup of $\G^+(\calP_1) \times \G^+(\calP_2)$ generated by the diagonal elements
		$\alpha_1 := (\s_1, \s_1')$ and $\alpha_2 := (\s_2, \s_2')$. Clearly,
		the order of each $\alpha_i$ is the least common multiple of the orders of $\s_i$ and $\s_i'$.
		We also define the mix of the polyhedra $\calP_1$ and $\calP_2$, denoted $\calP_1 \mix \calP_2$,
		to be the poset built from the group $\G^+(\calP_1) \mix \G^+(\calP_2)$ (using the construction
		in \cite{chiral}). Indeed, by \cite[Cor 3.8]{mix-ch}, the mix must itself be a polyhedron, yielding the following:
		
		\begin{proposition}
		\label{prop:polytopal-mix}
		Let $\calP_1$ be a chiral or orientably regular polyhedron of type $\{p_1, q_1\}$,
		and let $\calP_2$ be a chiral or orientably regular polyhedron of type $\{p_2, q_2\}$.
		Let $p$ be the least common multiple of $p_1$ and $p_2$, and let $q$ be the least common
		multiple of $q_1$ and $q_2$. Then $\calP_1 \mix \calP_2$ is a chiral or orientably regular
		polyhedron of type $\{p, q\}$.
		\end{proposition}
		
		Dual to the mix is the \emph{comix} of two groups. The comix of $\G^+(\calP_1)$ and
		$\G^+(\calP_2)$, denoted $\G^+(\calP_1) \comix \G^+(\calP_2)$, is the group we obtain
		from $\G^+(\calP_1)$ by adding all of the relations of $\G^+(\calP_2)$ (with each
		$\s_i'$ changed to $\s_i$). Its primary use is in determining the size of the mix,
		as the following result (\cite[Prop. 3.3]{mix-ch}) shows.
		
		\begin{proposition}
		\label{prop:mix-size}
		${\displaystyle |\G^+(\calP_1) \mix \G^+(\calP_2)| = \frac{|\G^+(\calP_1)| \cdot |\G^+(\calP_2)|}{|\G^+(\calP_1) \comix \G^+(\calP_2)|}.}$
		\end{proposition}
		
	\subsection{Tight polyhedra}
	
		See \cite{tight-polytopes, tight3} for information on tight polyhedra. We mention
		several key results here.
		A polyhedron of type $\{p, q\}$ must have at least $2pq$ flags, since it has
		at least $p$ vertices and the vertex-figures each have $2q$ flags. We call a polyhedron
		\emph{tight} if it has exactly $2pq$ flags. If $\calP$ is a tight chiral or orientably
		regular polyhedron of type $\{p, q\}$, then $|\G^+(\calP)| = pq$, half the number of flags.
		Furthermore, for any chiral or orientably regular polyhedron (not necessarily tight),
		the subset $\langle \s_1 \rangle \langle \s_2 \rangle$ of $\G^+(\calP)$ has
		$pq$ elements, because $\langle \s_1 \rangle$ and $\langle \s_2 \rangle$ have trivial intersection.
		Thus, if $\calP$ is a tight chiral or orientably regular polyhedron, it follows that
		$\G^+(\calP) = \langle \s_1 \rangle \langle \s_2 \rangle$.

		For convenience, we will say that a group
		$\langle \s_1, \s_2 \rangle$ is \emph{tight} if
		$\langle \s_1, \s_2 \rangle = \langle \s_1 \rangle \langle \s_2 \rangle$.
		If $\G$ is tight, $(\s_1 \s_2)^2 = \eps$ and $\langle \s_1 \rangle \cap \langle \s_2 \rangle = \{\eps\}$, 
		then $\G$ is the rotation group of a tight chiral or orientably regular polyhedron.
		
		Note that if $\G$ is tight, then so is any quotient of $\G$. For certain simple
		quotients, the converse is true:
		
		\begin{proposition}
		\label{prop:tight-crit}
		Let $\G = \langle \s_1, \s_2 \rangle$, and suppose that there is a normal subgroup
		$N = \langle \s_2^k \rangle$ such that $\G/N$ is tight.
		Then $\G$ is tight.
		\end{proposition}
		
		\begin{proof}
		Let $w \in \G$, and let $\ol{w}$ be the image of $w$ in $\G/N$. Since
		$\G/N$ is tight, we can write $\ol{w} = \ol{\s_1}^a \ol{\s_2}^b$ for some $a$ and $b$.
		It follows that $w = \s_1^a \s_2^b \s_2^{ck}$ for some $c$, and so $w \in \langle \s_1 \rangle
		\langle \s_2 \rangle$.
		\end{proof}

		\begin{proposition}
		\label{prop:inverse-quo-crit}
		Suppose $\calP$ is a tight chiral or orientably regular polyhedron of type $\{p, q\}$ with $p \geq 3$.
		Let $\G^+(\calP) = \langle \s_1, \s_2 \rangle$, and let $N = \langle \s_2^{q'} \rangle$.
		If $N$ is normal in $\G^+(\calP)$, then $\G^+(\calP) / N$ is the rotation group of
		a tight chiral or orientably regular polyhedron. 
		\end{proposition}
		
		\begin{proof}
		Let $\G^+(\calP) / N = \langle \ol{\s_1}, \ol{\s_2} \rangle$.
		First we want to show that $\G^+(\calP) / N$ is the rotation group of a chiral or orientably
		regular polyhedron. Suppose that $q' = 1$. Then $\s_1 \s_2 \s_1^{-1} = \s_2^a$ for some $a$,
		and since $(\s_1 \s_2)^2 = \eps$, it follows that $\s_2^{-1} \s_1^{-2} = \s_2^a$ and thus
		$\s_1^{-2} = \s_2^{a+1}$. Then since $\langle \s_1 \rangle$ and $\langle \s_2 \rangle$
		must have trivial intersection, we see that $\s_1^2 = \eps$.
		Thus, under the assumption that $p \geq 3$ we have $q' \geq 2$. 
		Then by \cite[Theorem 1]{chiral}, all we need to show is that $\langle \ol{\s_1} \rangle \cap
		\langle \ol{\s_2} \rangle = \{\eps\}$. Consider an element in this intersection;
		it must be $\ol{w}$ for some $w \in \G^+(\calP)$. Since $\ol{w} \in \langle \ol{\s_1} \rangle$,
		it follows that $w = \s_1^i \s_2^{aq'}$ for some integers $i$ and $a$.
		Similarly, since $\ol{w} \in \langle \ol{\s_2} \rangle$, it follows that $w = \s_2^j \s_2^{bq'}$
		for some integers $j$ and $b$. Therefore $\s_1^i = \s_2^{j+bq'-aq'}$. Since
		$\calP$ is a polyhedron, $\langle \s_1 \rangle \cap \langle \s_2 \rangle = \{\eps\}$,
		and so $i = 0$. Thus $w = \s_2^{aq'}$, and so $\ol{w} = \eps$, which is what we wanted to show.
		Finally, since $\calP$ is tight, it follows that $\G^+(\calP) = \langle \s_1 \rangle \langle \s_2 \rangle$,
		and thus $\G^+(\calP) / N = \langle \ol{\s_1} \rangle \langle \ol{\s_2} \rangle$.
		Therefore, $\G^+(\calP) / N$ is tight.
		\end{proof}
		
		\begin{proposition}
		\label{prop:tight-covers-chiral}
		Suppose that $\calP$ is a tight chiral or orientably regular polyhedron, and that
		$\calP$ covers a tight chiral polyhedron. Then $\calP$ is itself chiral.
		\end{proposition}
		
		\begin{proof}
		Suppose that $\calP$ is regular. Let $\calQ$ be a tight chiral polyhedron of type $\{p, q\}$ covered by $\calP$,
		and let $\calR$ be the minimal regular cover of $\calQ$. Since $\calQ$ and $\calR$ both have type $\{p, q\}$ 
		and $\calQ$ is a proper quotient of $\calR$, it follows that $|\G^+(\calR)| > pq$. So $\G^+(\calR)$ is
		not tight, and it follows that neither is $\G^+(\calP)$. So if $\G^+(\calP)$ (and therefore $\calP)$
		is tight, then $\calP$ is chiral.
		\end{proof}
		
\section{Families of tight chiral polyhedra}

	\subsection{Known tight chiral polyhedra}

	Our search for tight chiral polyhedra begins with Marston Conder's list of chiral 
	polytopes with up to 2000 flags \cite{chiral-atlas}. Every tight chiral polyhedron 
	with at most 2000 flags either has its Schl\"afli symbol or the dual of its Schl\"afli 
	symbol in Table 1. In every case where the parameter $n$ appears, its 
	upper bound is simply what is required to ensure that the polyhedron has at most 2000 flags.

	\begin{table}[h!]
	\begin{center}
	\begin{tabular}{l | l}
	$\{6, 9n\}$ for $1 \leq n \leq 18$ 	&	$\{8, 32n\}$ for $1 \leq n \leq 3$ \\
	$\{9, 18\}$							&	$\{10, 25n\}$ for $1 \leq n \leq 4$ \\
	$\{12, 18n\}$ for $1 \leq n \leq 4$	&	$\{14, 49\}$ \\
	$\{16, 32\}$						&	$\{18, 6n\}$ for $3 \leq n \leq 9$ \\
	$\{18, 9n\}$ for $2 \leq n \leq 6$	&	$\{20, 50\}$ \\
	$\{24, 32\}$						&	$\{24, 36\}$.
	\end{tabular}
	\end{center}
	\caption{Schl\"afli symbols of tight chiral polyhedra with at most 2000 flags}
	\label{schlafli}
	\end{table}
	
	There are several interesting patterns in the data. In all of the Schl\"afli symbols,
	at least one of the numbers is divisible by a nontrivial square. Closer examination 
	reveals that all of the Schl\"afli symbols are a ``multiple'' of $\{8, 32\}$ or
	of $\{2m, m^2\}$ or $\{m^2, 2m\}$ for an odd prime $m$.
	That is, every Schl\"afli symbol in Table~\ref{schlafli} can be written as $\{8r, 32s\}$ or
	$\{2rm, sm^2\}$ for an odd prime $m$, or $\{sm^2, 2rm\}$ for an odd prime $m$.
	This pattern suggests that the tight chiral polyhedra of types $\{8, 32\}$ and $\{2m, m^2\}$ 
	play a fundamental role. We will see later that this is indeed the case.

	\subsection{Rotation groups of tight chiral polyhedra}
	
	If $\calP$ is a tight chiral or orientably regular polyhedron, then every element of
	$\G^+(\calP)$ has an essentially unique representation of the form $\s_1^i \s_2^j$.
	In particular, $\s_2^{-1} \s_1$ and $\s_2 \s_1^{-1}$ both have representations
	of this form, and in many cases, knowing how to represent these
	two elements of the group is already enough to define the entire group.
	Thus, we define the group $\GP{p,q}{i_1,j_1,i_2,j_2}$ as:
	\begin{align}
	\begin{split}
	\GP{p,q}{i_1, j_1, i_2, j_2} := \langle \s_1, \s_2 \mid & \s_1^p = \s_2^q = (\s_1 \s_2)^2 = \eps, \\
	& \s_2^{-1} \s_1 = \s_1^{i_1} \s_2^{j_1}, \\
	& \s_2 \s_1^{-1} = \s_1^{i_2} \s_2^{j_2}.
	\end{split}
	\end{align}
	Note that the dual of $\GP{p,q}{i_1, j_1, i_2, j_2}$ is $\GP{q, p}{j_2, i_2, j_1, i_1}$.
	(That is, if we change the relations of the former by sending each $\s_i$ to $\s_{3-i}^{-1}$,
	we get the latter.)
	Also, the enantiomorphic form of $\GP{p,q}{i_1, j_1, i_2, j_2}$ is $\GP{p,q}{-i_2, -j_2, -i_1, -j_1}$.
	
	We start by collecting some basic facts about these groups.
	
	\begin{proposition}
	\label{prop:gp-props2}
	\begin{enumerate}
	\item In $\GP{p, q}{i, j_1, -i, j_2}$, the subgroups $\langle \s_2^{j_1-1} \rangle$ and 
		$\langle \s_2^{j_2+1} \rangle$ are identical and normal.
	\item In $\GP{p, q}{i_1, j, i_2, -j}$, the subgroups $\langle \s_1^{i_1+1} \rangle$ and 
		$\langle \s_1^{i_2-1} \rangle$ are identical and normal.
	\item The group $\GP{p, q}{-1, 1, 1, -1}$ is tight.
	\item For any $i$, the group $\GP{p, q}{i, 1, -i, -1}$ is tight.
	\item For any $i$, $j_1$, and $j_2$, the group $\GP{p, q}{i, j_1, -i, j_2}$ is tight.
	\end{enumerate}
	\end{proposition}
	
	\begin{proof}
	For part (a), we find:
	\begin{align*}
	\s_1 \s_2^{j_1-1} &= \s_1 \s_2^{-1} \s_2^{j_1} \\
	&= \s_2^{-j_2} \s_1^i \s_2^{j_1} \\
	&= \s_2^{-j_2-1} \s_1.
	\end{align*}
	Thus, $\langle \s_2^{j_1-1} \rangle$ is normal and identical to $\langle \s_2^{j_2+1} \rangle$.
	Part (b) follows by a dual argument.

	For part (c), we note that in $\GP{p, q}{-1, 1, 1, -1}$, we have the relations $\s_2^{-1} \s_1 = \s_1^{-1} \s_2$
	and $\s_2 \s_1^{-1} = \s_1 \s_2^{-1}$.
	Using this and the standard relation $\s_2 \s_1 = \s_1^{-1} \s_2^{-1}$, we find that
	$\s_2^a \s_1 = \s_1^{(-1)^a} \s_2^{-a}$ for any $a$, and therefore 
	$\s_2^a \s_1^b = \s_1^{b(-1)^a} \s_2^{a(-1)^b}$ for any $a$ and $b$. It follows that we
	can rewrite any element of $\GP{p, q}{-1, 1, 1, -1}$ as the product of a power of $\s_1$
	with a power of $\s_2$, and so this group is tight.

	To prove part (d), we start by noting that $\GP{p, q}{i, 1, -i, -1}$ has normal subgroup 
	$\langle \s_1^{i+1} \rangle$, by part (b). The quotient by this subgroup is
	$\GP{p, q}{-1, 1, 1, -1}$, which is tight by
	part (c). Then \pref{tight-crit} implies that $\GP{p, q}{i, 1, -i, -1}$ is tight.
	Similarly, $\GP{p, q}{i, j_1, -i, j_2}$ has normal subgroup $\langle \s_2^{j_1-1} \rangle =
	\langle \s_2^{j_2 + 1} \rangle$ (by part (a)), and the quotient is $\GP{p, q}{i, 1, -i, -1}$.
	Applying \pref{tight-crit} again proves that $\GP{p, q}{i, j_1, -i, j_2}$ is tight, proving
	part (e).
	\end{proof}

	Note that in the group $\GP{p,q}{i_1, j_1, i_2, j_2}$, the orders of $\s_1$ and $\s_2$
	could in principle collapse to proper divisors of $p$ and $q$, respectively. 
	
	A few examples of these groups have been previously studied during the
	classification of tight regular polyhedra, and they will be useful to
	us shortly.
	
	\begin{proposition}
	\label{prop:regular-gp1}
	If $q$ is odd and $p$ is an even divisor of $2q$, then there is a unique tight orientably
	regular polyhedron of type $\{p, q\}$ (up to isomorphism), and its rotation group
	is $\GP{p, q}{3, 1, -3, -1}$.
	\end{proposition}
	
	\begin{proof}
	Theorems 3.1 and 3.3 in \cite{tight2} demonstrate that there is a unique tight
	orientably regular polyhedron $\calP$ of type $\{r, s\}$ whenever $r$ is odd and $s$ is an
	even divisor of $2r$. The full automorphism group
	of $\calP$ is described there as
	\begin{align*}
	\langle \rho_0, \rho_1, \rho_2 &\mid \rho_0^2 = \rho_1^2 = \rho_2^2 = \eps, \\
	& (\rho_0 \rho_1)^r = (\rho_0 \rho_2)^2 = (\rho_1 \rho_2)^s = \eps, \\
	& (\rho_0 \rho_1 \rho_2 \rho_1 \rho_2)^2 = \eps \rangle.
	\end{align*}
	Setting $\s_1 = \rho_0 \rho_1$ and $\s_2 = \rho_1 \rho_2$, we see that the relation 
	$\s_1 \s_2^{-1} \s_1 \s_2^3 = \eps$ holds in $\G^+(\calP)$, 
	from which it follows that $\s_2^{-1} \s_1 = \s_1^{-1} \s_2^{-3}$. Conjugating this by $\rho_1$ shows
	that also $\s_2 \s_1^{-1} = \s_1 \s_2^3$. So $\G^+(\calP)$ is a quotient
	of $\GP{r, s}{-1, -3, 1, 3}$, and since this group is already tight (by \pref{gp-props2}(e)), it follows
	that $\G^+(\calP) = \GP{r, s}{-1, -3, 1, 3}$. It follows that the dual of $\calP$ is the unique
	tight orientably regular polyhedron of type $\{s, r\}$, and that it has group
	$\GP{s, r}{3, 1, -3, 1}$. Taking $s = p$ and $r = q$ then proves our claim.
	\end{proof}
	
	We will also need the following consequence of \cite[Thm. 4.13]{tight3}:
	
	\begin{proposition}
	\label{prop:regular-gp2}
	If $\alpha \geq 4$, then $\GP{2^{\alpha}, 4}{-1 + 2^{\alpha-1}, 1, 1 - 2^{\alpha-1}, -1}$
	is the rotation group of a tight orientably regular polyhedron of
	type $\{2^{\alpha}, 4\}$, and $\GP{2^{\alpha}, 2^{\alpha-1}}{3,1,-3,-1}$ is the
	rotation group of a tight orientably regular polyhedron of type $\{2^{\alpha}, 2^{\alpha-1}\}$.
	\end{proposition}
	
	Finally, we need the following consequence of \cite[Lemma 6.2, Thm. 6.3]{tight-polytopes}:
	
	\begin{proposition}
	\label{prop:regular-gp3}
	If $p$ and $q$ are even, then $\GP{p, q}{-1, 1, 1, -1}$ is the rotation group of a tight
	orientably regular polyhedron of type $\{p, q\}$.
	\end{proposition}
	
	Our first goal will be to describe three families of tight chiral polyhedra that,
	we will see later, are particularly important. Most of these polyhedra will
	have automorphism groups $\GP{p,q}{i_1, j_1, i_2, j_2}$. 
	Given such a group, we will often need to show three things: that it is the
	automorphism group of a polyhedron $\calP$, that $\calP$ is tight, and
	that $\calP$ is chiral. In many cases, we are able to use the following ``bootstrapping'' lemma:
	
	\begin{lemma}
	\label{lem:lifting}
	Let $\G = \GP{p, q}{i, j_1, -i, j_2}$, and suppose that $\s_2$ has order $q$ in $\G$.
	Suppose that for some $q'$ dividing $q$, $\G$ covers the group $\GP{p, q'}{i, j_1, -i, j_2}$,
	and that $\GP{p, q'}{i, j_1, -i, j_2}$ is the rotation group of a tight
	chiral or orientably regular polyhedron of type $\{p, q'\}$. If 
	$j_1 \not \equiv -j_2$ (mod $q$), then $\G$ is the automorphism group
	of a tight chiral polyhedron of type $\{p, q\}$.
	\end{lemma}

	\begin{proof}
	First, \pref{gp-props2}(e) implies that $\G$ is tight.
	Since $\GP{p, q'}{i, j_1, -i, j_2}$ is the rotation group
	of a tight chiral or orientably regular polyhedron of type $\{p, q'\}$, it has order
	$pq'$. Similarly, since $\G$ is tight and $\s_1^p = \s_2^q = \eps$ it follows
	that $\G$ has order at most $pq$. Since $\s_2$ has order $q$ in $\G$ and $q'$ divides $q$,
	it follows that $\s_2^{q'}$ has order $q/q'$ in $\G$. Then the kernel of the covering
	from $\G$ to $\GP{p,q'}{i, j_1, -i, j_2}$ cannot be any larger than $\langle \s_2^{q'} \rangle$,
	and so $\langle \s_2^{q'} \rangle$ is normal in $\G$.
	Then the \emph{quotient criterion} \cite[Lemma 3.2]{chiral-mix} implies that $\G$
	is the rotation group of a tight chiral or orientably regular polyhedron $\calP$
	of type $\{p, q\}$.
	
	Suppose that $\calP$ is regular. Then $\G^+(\calP)$ has a group automorphism that sends
	each $\s_i$ to $\s_i^{-1}$, and so from the relation
	$\s_2 \s_1^{-1} = \s_1^{-i} \s_2^{j_2}$, it follows that $\s_2^{-1} \s_1 =
	\s_1^{i} \s_2^{-j_2}$. Combining with the relation $\s_2^{-1} \s_1 = \s_1^{i} \s_2^{j_1}$,
	we get that $\s_2^{j_1} = \s_2^{-j_2}$, and since $\s_2$ has order $q$, it follows
	that $j_1 \equiv -j_2$ (mod $q$). So if $j_1 \not \equiv -j_2$, then $\calP$ is chiral.
	\end{proof}
	
	\subsection{First three families of tight chiral polyhedra}
	
	In each of the next three theorems, we give a presentation for a family of groups,
	and it will be clear that the groups cover one of the ones in \pref{regular-gp1}
	or \pref{regular-gp2}. In light of \lref{lifting}, all that then remains
	is to show that $\s_2$ has the correct order. To do so, we will build a
	permutation representation of the given group. 
	Indeed, the representation we use is simply the action of the group on
	the cosets of $\langle \s_1 \rangle$; this was used to determine the
	proper definitions of the permutations $\pi_1$ and $\pi_2$, but we do not
	rely on this fact for the proofs, nor do we prove that the permutation representation is faithful.

	\begin{theorem}
	\label{thm:odd-atomics}
	For every odd prime $m$, positive integer ${\beta} \geq 2$, and integer $k$ satisfying $1 \leq k \leq m-1$,
	the group 
	\[ \GP{2m, m^{\beta}}{3, 1+km^{\beta-1}, -3, -1+km^{\beta-1}} \]
	is the automorphism group of a tight chiral polyhedron of type $\{2m, m^{\beta}\}$.
	\end{theorem}
	
	\begin{proof}
	Let $\G := \GP{2m, m^{\beta}}{3, 1+km^{\beta-1}, -3, -1+km^{\beta-1}}$. 
	Then $\G$ covers $\GP{2m, m^{\beta-1}}{3, 1, -3, -1}$ (by inspection of the presentations), 
	which by \pref{regular-gp1} is
	the rotation group of a tight orientably regular polyhedron of type $\{2m, m^{\beta-1}\}$.
	In light of \lref{lifting}, 
	all that remains is to show that $\s_2$ has order $m^{\beta}$.
	To do so, we provide a permutation representation of
	$\G$ on $\Z_{m^{\beta}}$. To simplify the representation, we will actually provide
	a permutation representation of $\GP{2m, m^{\beta}}{3, 1-2km^{\beta-1}, -3, -1-2km^{\beta-1}}$
	(that is, with $k$ changed to $-2k$); since $m$ is an odd prime and $1 \leq k \leq m-1$,
	this defines the same set of groups.
	We define functions $\pi_1$ and $\pi_2$ on $\Z_{m^{\beta}}$ by
	\[ b \pi_1 = - b + b(1-b) k m^{{\beta}-1} \]
	\[ b \pi_2 = b + 1. \]
	First of all, we need to demonstrate that $\pi_1$ is actually a permutation. 
	(It is obvious that $\pi_2$ is.) An easy calculation shows that $b \pi_1^2 = 
	b(1-2km^{{\beta}-1})$. Thus, for each $n$,
	\[ b \pi_1^{2n} = b(1-2km^{{\beta}-1})^n = b(1-2nkm^{{\beta}-1}), \]
	since we are working modulo $m^{\beta}$.
	In particular, $b \pi_1^{2m} = b(1-2km^{\beta}) = b$. So a finite power of $\pi_1$ is the identity,
	which implies that it is a permutation.
	It is now straightforward to check that $\langle \pi_1, \pi_2 \rangle$ satisfies all of the defining
	relations of $\GP{2m, m^{\beta}}{3, 1-2km^{\beta-1}, -3, -1-2km^{\beta-1}}$ 
	(with each $\s_i$ replaced by $\pi_i$), and so $\langle \pi_1, \pi_2 \rangle$
	really is a permutation representation of this group. It is clear that $\pi_2$ has order $m^{\beta}$,
	and so $\s_2$ does as well. So $\G$ is the automorphism group of a tight chiral 
	polyhedron of type $\{2m, m^{\beta}\}$.
	\end{proof}
	
	\begin{theorem}
	\label{thm:even-atomics}
	For each positive integer ${\beta} \geq 5$, the groups 
	\[ \GP{8, 2^{\beta}}{3, 1-2^{\beta-2}, -3, -1-2^{\beta-2}} \] 
	and
	\[ \GP{8, 2^{\beta}}{3, 1+2^{\beta-2}, -3, -1+2^{\beta-2}} \] 
	are the automorphism groups of tight chiral polyhedra of type $\{8, 2^{\beta}\}$. 
	\end{theorem}

	\begin{proof}
	We prove the result for $\G = \GP{8, 2^{\beta}}{3, 1-2^{\beta-2}, -3, -1-2^{\beta-2}}$;
	the second group is the enantimorphic form of the first and will thus follow.
	The group $\G$ covers $\GP{8, 4}{3, 1, -3, -1}$,
	and a calculation with GAP \cite{gap} shows that this is the rotation group
	of a tight orientably regular polyhedron of type $\{8, 4\}$. 
	By \lref{lifting}, we will be done if we can show that $\s_2$ has order
	$2^{\beta}$. 
	We provide a permutation representation of $\G$ on $\Z_{2^{\beta}}$,
	defining functions $\pi_1$ and $\pi_2$ as follows:
	\[ b \pi_1 = -b + b(1-b) 2^{{\beta}-3}, \]
	\[ b \pi_2 = b+1. \]
	It is straightforward to show that $\pi_1^2$ sends $b$ to 
	$b(1-2^{{\beta}-2})$, and then that $\langle \pi_1, \pi_2 \rangle$ 
	satisfies all of the defining relations of $\G$. It follows that 
	$\langle \pi_1, \pi_2 \rangle$ really is a permutation representation of $\G$, 
	and since $\pi_2$ clearly has order $2^{\beta}$, so does $\s_2$, and the result follows.
	\end{proof}

	\begin{theorem}
	\label{thm:even-atomics2}
	For each positive integer $\beta \geq 5$, the groups
	\[ \GP{2^{\beta-1}, 2^{\beta}}{-1 + 2^{\beta-2}, -3 + 2^{\beta-2}, 1 - 2^{\beta-2}, 3 + 2^{\beta-2}} \]
	and
	\[ \GP{2^{\beta-1}, 2^{\beta}}{-1 + 2^{\beta-2}, -3 - 2^{\beta-2}, 1 - 2^{\beta-2}, 3 - 2^{\beta-2}} \]
	are the automorphism groups of tight chiral polyhedra of type
	$\{2^{\beta-1}, 2^{\beta}\}$.
	\end{theorem}
	
	\begin{proof}
	We prove the result for 
	$\G = \GP{2^{\beta-1}, 2^{\beta}}{-1 + 2^{\beta-2}, -3 + 2^{\beta-2}, 1 - 2^{\beta-2}, 3 + 2^{\beta-2}}$;
	the second group is the enantimorphic form of the first and will thus follow.
	The group $\G$ covers 
	$\GP{2^{\beta-1}, 4}{-1 + 2^{\beta-2}, 1, 1 - 2^{\beta-2}, -1}$, which is
	the rotation group of a tight orientably regular polyhedron of type $\{2^{\beta-1}, 4\}$,
	by \pref{regular-gp2}. Then by \lref{lifting}, all that remains is to show that
	$\s_2$ has order $2^{\beta}$. 

	We construct a permutation representation of $\G$.
	Let us define permutations $\pi_1$ and $\pi_2$ on $\Z_{2^{\beta}}$ by
	\[ b \pi_1 = \begin{cases}
		b + 2^{\beta-3} b(b-1), \textrm{ if $b$ is even,} \\
		b - 2 + 2^{\beta-3} b(b-1), \textrm{ if $b$ is odd;}
		\end{cases}
	\]
	\[ b \pi_2 = b + 1. \]
	It can be shown that $\pi_1$ is invertible, with
	\[ b \pi_1^{-1} = \begin{cases}
		b - 2^{\beta-3} b(b-1), \textrm{ if $b$ is even,} \\
		b + 2 + 2^{\beta-3} b(b+1), \textrm{ if $b$ is odd.}
	\end{cases}
	\] 
	It follows that $\pi_1$ and $\pi_2$ are well-defined permutations.

	Now we need to show that there is a well-defined homomorphism
	sending each $\s_i$ to $\pi_i$, for which it suffices to show that
	$\langle \pi_1, \pi_2 \rangle$ satisfies the defining relations of
	$\G$ when we change each $\s_i$ to $\pi_i$. The calculations here are
	a little more involved than in \tref{odd-atomics} and \tref{even-atomics},
	so we show some of the details.
	We start by calculating $\pi_1^2$ and then $\pi_1^4$. First, suppose
	that $b$ is even. If we set $b' = b \pi_1$, then we note that
	\[ 2^{\beta-3} b' = 2^{\beta-3} b + 2^{2 \beta-6} b(b-1) \equiv 2^{\beta-3} b \textrm{ (mod $2^{\beta}$) }, \]
	since $\beta \geq 5$ and $b(b-1)$ must be even. Then
	\begin{align*}
	b \pi_1^2 &= b' \pi_1 \\
	&= b' + 2^{\beta-3} b'(b'-1) \\
	&= b' + 2^{\beta-3} b(b-1) \\
	&= b + 2^{\beta-2} b(b-1). 
	\end{align*}
	A similar calculation shows that when $b$ is even, 
	\begin{equation}
	\label{eq:b-even}
	b \pi_1^4 = b + 2^{\beta-1} b(b-1) = b.
	\end{equation}
	Now, suppose that $b$ is odd instead. Again, let us set
	$b' = b \pi_1$; in this case we get that
	\[ 2^{\beta-3} b' = 2^{\beta-3}(b-2) + 2^{2 \beta-6} b(b-1) \equiv 2^{\beta-3} (b-2) \textrm{ (mod $2^{\beta}$) }. \]
	Therefore,
	\begin{align*}
	b \pi_1^2 &= b' \pi_1 \\
	&= b' - 2 + 2^{\beta-3} b'(b'-1) \\
	&= b' - 2 + 2^{\beta-3} (b-2)(b-3) \\
	&= b - 4 + 2^{\beta-3} [b(b-1) + (b-2)(b-3)]. 
	\end{align*}
	Continuing in this manner, we find that
	\begin{align}
	\label{eq:b-odd}
	\begin{split}
	b \pi_1^4 &= b - 8 + 2^{\beta-3} [b(b-1) + (b-2)(b-3) + (b-4)(b-5) + (b-6)(b-7)] \\
	&= b - 8 + 2^{\beta-3} [4b^2 - 28b + 68] \\
	&= b - 8 + 2^{\beta-1} [b^2 - 7b + 17] \\
	&= b - 8 + 2^{\beta-1}, 
	\end{split}
	\end{align}
	where the last line follows because $b^2 - 7b + 17$ is odd and we are working modulo $2^{\beta}$.
	Combining Equations \ref{eq:b-even} and \ref{eq:b-odd}, we get that
	\begin{equation}
	b \pi_1^8 = \begin{cases}
		b, \textrm{ if $b$ is even,} \\
		b - 16, \textrm{ if $b$ is odd.}
	\end{cases}
	\end{equation}
	With that, it is straightforward to check that $\langle \pi_1, \pi_2 \rangle$ does indeed
	satisfy the relations of $\G$, and so $\G$ covers $\langle \pi_1, \pi_2 \rangle$.
	Since $\pi_2$ clearly has order $2^{\beta}$, it follows that $\s_2$ does as well,
	and the result is proven.
	\end{proof}

	\subsection{Five more families of tight chiral polyhedra}
	
	Using the polyhedra from Theorems \ref{thm:odd-atomics}, \ref{thm:even-atomics}, and
	\ref{thm:even-atomics2} as our foundation, we can now lift them to construct
	more examples.

	\begin{theorem}
	\label{thm:odd-centrals}
	Let $m$ be an odd prime, let $\alpha$ and $\beta$ be integers with $\alpha \geq 1$
	and $\beta \geq 2$, and let
	$k$ be an integer satisfying $1 \leq k \leq m-1$. Then the group
	\begin{align*}
		\langle \s_1, \s_2 \mid & \s_1^{2m^{\beta}} = \s_2^{m^{\beta}} = (\s_1 \s_2)^2 = 1; \\
						  & \s_2^{-1} \s_1 = \s_1^{3+k(m+1)m^{\beta-1}} \s_2^{1+km^{\beta-1}}; \\
						  & \s_2 \s_1^{-1} = \s_1^{-3+k(m+1)m^{\beta-1}} \s_2^{-1+km^{\beta-1}}; \\
						  & \s_1^{2m} \s_2 = \s_2 \s_1^{2m} \rangle
	\end{align*}
	is the automorphism group of a tight chiral polyhedron of type $\{2 m^{\beta}, m^{\beta}\}$.
	If $\beta > \alpha$, then the group
	\begin{align*}
	\langle \s_1, \s_2 \mid & \s_1^{2m^{\alpha}} = \s_2^{m^{\beta}} = (\s_1 \s_2)^2 = 1; \\
					  & \s_2^{-1} \s_1 = \s_1^3 \s_2^{1+km^{\beta-1}}; \\
					  & \s_2 \s_1^{-1} = \s_1^{-3} \s_2^{-1+km^{\beta-1}}; \\
					  & \s_1^{2m}\s_2 = \s_2 \s_1^{2m} \rangle
	\end{align*}
	is the automorphism group of a tight chiral polyhedron of type $\{2 m^{\alpha}, m^{\beta}\}$.
	\end{theorem}
	
	\begin{proof}
	First, we show that $\s_1$ has the desired order.
	In the first case, we can check that there is an epimorphism to the 
	cyclic group $\langle x \mid x^{2m^{\beta}} = 1 \rangle$ that sends 
	$\s_1$ to $x$ and $\s_2$ to $x^{m^{\beta}-1}$. 
	In the second case, consider $\GP{2m^{\alpha}, m^{\beta-1}}{3, 1, -3, -1}$.
	In this group,
	\[ \s_2^{-1} \s_1 = \s_1^3 \s_2 = \s_1^2 \s_2^{-1} \s_1^{-1}, \]
	and it follows that $\s_1^2$ commutes with $\s_2$. So $\s_1^2$ is central, and thus
	the group given in the second case covers $\GP{2m^{\alpha}, m^{\beta-1}}{3, 1, -3, -1}$.
	By \pref{regular-gp1}, this is the group of a tight orientably regular polyhedron of type
	$\{2m^{\alpha}, m^{\beta-1}\}$. Thus in both cases, $\s_1$ has the desired order.
	
	Now, both groups have $\s_1^{2m}$ central, and they both cover $\GP{2m, m^{\beta}}{3, 1+km^{\beta-1}, -3, -1+km^{\beta-1}}$.
	By \tref{odd-atomics}, this is the group of a tight chiral polyhedron $\calQ$
	of type $\{2m, m^{\beta}\}$. The the quotient criterion \cite[Lemma 3.2]{chiral-mix} combined with \pref{tight-crit}
	implies that the given groups are the rotation groups of tight chiral or orientably regular polyhedra.
	In fact, since both polyhedra cover the tight chiral polyhedron $\calQ$,
	\pref{tight-covers-chiral} implies that the polyhedra must be chiral. Finally, it is now clear that
	$\s_2$ has order $m^{\beta}$ in each rather than a proper divisor. 
	\end{proof}

	We remark here that in the second case, it was essential to the argument that $\beta > \alpha$. The first case
	covered the possibility $\beta = \alpha$, and we will see later that if $\beta < \alpha$, then
	there are no tight chiral polyhedra of type $\{2m^{\alpha}, m^{\beta}\}$. (For example, there
	are no tight chiral polyhedra of type $\{54, 9\}$, which can be verified by looking through \cite{chiral-atlas}.)

	\begin{theorem}
	\label{thm:even-centrals}
	Let $\beta \geq \alpha + 1$, with $\alpha \geq 3$ and $\beta \geq 5$.
	Then the groups 
	\[ \GP{2^{\alpha}, 2^{\beta}}{3, 1+2^{\beta-2}, -3, -1+2^{\beta-2}} \]
	and 
	\[ \GP{2^{\alpha}, 2^{\beta}}{3, 1-2^{\beta-2}, -3, -1-2^{\beta-2}} \]
	are the automorphism
	groups of tight chiral polyhedra of type $\{2^{\alpha}, 2^{\beta}\}$.
	\end{theorem}
	
	\begin{proof}
	We proceed by induction on $\alpha$. \tref{even-atomics} proves the base
	case $\alpha = 3$. In the general case, the group $\GP{2^{\alpha}, 2^{\beta}}{3, 1\pm 2^{\beta-2},
	-3, -1 \pm 2^{\beta-2}}$ covers $\GP{2^{\alpha-1}, 2^{\beta}}{3, 1\pm 2^{\beta-2},
	-3, -1 \pm 2^{\beta-2}}$, which is the automorphism group of a tight chiral polyhedron
	of type $\{2^{\alpha-1}, 2^{\beta}\}$ by inductive hypothesis. It follows that $\s_2$ has
	order $2^{\beta}$. The group $\GP{2^{\alpha}, 2^{\beta}}{3, 1\pm 2^{\beta-2},
	-3, -1 \pm 2^{\beta-2}}$ also covers $\GP{2^{\alpha}, 2^{\alpha-1}}{3, 1, -3, -1}$, which
	is the rotation group of a tight orientably regular polyhedron of type
	$\{2^{\alpha}, 2^{\alpha-1}\}$, by \pref{regular-gp2}. Thus $\s_1$ has order $2^{\alpha}$.
	The result then follows from \lref{lifting}.
	\end{proof}

	In order to construct the final 3 families of tight chiral polyhedra, we will
	use the mixing construction. We start with a lemma.
	
	\begin{lemma}
	\label{lem:mixing-tight}
	Suppose that $\calP_1$ is a tight chiral polyhedron of type $\{p_1, q_1\}$
	and that $\calP_2$ is a tight chiral or orientably regular polyhedron of type $\{p_2, q_2\}$.
	Let $p$ be the least common multiple of $p_1$ and $p_2$, and let $q$ be the least common
	multiple of $q_1$ and $q_2$. If the order of $\G^+(\calP_1) \mix \G^+(\calP_2)$ is
	$pq$, then $\calP_1 \mix \calP_2$ is a tight chiral polyhedron of type $\{p, q\}$.
	\end{lemma}
	
	\begin{proof}
	Let $\calQ = \calP_1 \mix \calP_2$. By \pref{polytopal-mix}, $\calQ$ is a polyhedron
	of type $\{p, q\}$. Furthermore, 
	\[ |\G^+(\calQ)| = |\G^+(\calP_1) \mix \G^+(\calP_2)| = pq, \]
	by supposition. So $\calQ$ has $2pq$ flags and is thus tight. Finally, since $\calQ$ is
	tight and it covers the chiral polyhedron $\calP_1$, it follows from \pref{tight-covers-chiral}
	that $\calQ$ is itself chiral.
	\end{proof}

	Using \lref{mixing-tight}, the proofs of the remaining three theorems will be reduced
	to demonstrating that the mix of two chosen groups has the right size.
	
	\begin{theorem}
	\label{thm:even-centrals2}
	For each $\alpha \geq 5$, the mix of
	\[ \G_1 = \GP{2^{\alpha-1}, 2^{\alpha}}{-1 + 2^{\alpha-2}, -3 + 2^{\alpha-2}, 1 + 2^{\alpha-2}, 3 + 2^{\alpha-2}}, \]
	with
	\[ \G_2 = \GP{2^{\alpha}, 8}{-1 + 2^{\alpha-2}, -3, 1 + 2^{\alpha-2}, 3} \]
	is the automorphism group of a tight chiral polyhedron of type $\{2^{\alpha}, 2^{\alpha}\}$.
	\end{theorem}
	
	\begin{proof}
	By \tref{even-atomics2}, $\G_1$ is the automorphism group of a tight chiral polyhedron of type
	$\{2^{\alpha-1}, 2^{\alpha}\}$, and by the dual of \tref{even-atomics}, $\G_2$ is the automorphism
	group of a tight chiral polyhedron of type $\{2^{\alpha}, 8\}$. In light of \lref{mixing-tight},
	all we need to show is that $\G_1 \mix \G_2$ has order $2^{2\alpha}$.
	To do so, we examine $\G_1 \comix \G_2$. Inspection of the presentations makes it clear that
	\[ \G_1 \comix \G_2 = \GP{2^{\alpha-1}, 8}{-1+2^{\alpha-2}, -3, 1+2^{\alpha-2}, 3}. \]
	By \pref{gp-props2}(a), the subgroup $\langle \s_2^4 \rangle$ is normal, and so $\G_1 \comix \G_2$
	covers $\GP{2^{\alpha-1}, 4}{-1+2^{\alpha-2}, 1, 1+2^{\alpha-2}, -1}$. By \pref{regular-gp2}, this
	latter is the rotation group of a tight orientably regular polyhedron of type $\{2^{\alpha-1}, 4\}$.
	It follows from \pref{tight-crit} that $\G_1 \comix \G_2$ is tight, and thus has order $2^{\alpha+2}$.
	Since $|\G_1| = 2^{2\alpha-1}$ and $|\G_2| = 2^{\alpha+3}$, the result then follows from \pref{mix-size}.
	\end{proof}
	
	\begin{theorem}
	\label{thm:odd-chirals}
	Let $m$ be an odd prime, and suppose that $r$ and $s$ are positive integers such that either
	$s$ is even, or $s$ is odd and $r$ divides $sm$. 
	Then there is a tight chiral polyhedron of type $\{2rm, s m^2\}$.
	\end{theorem}
	
	\begin{proof}
	Let us write $r = m^{\alpha} r'$ and $s = m^{\beta} s'$, with $r'$ and $s'$ both coprime to $m$.
	We consider 4 cases. \\
	
	\par\noindent \textbf{Case 1}: $s$ is even and $\beta \geq \alpha$. Let
	\[ \G_1 = \GP{2m^{\alpha+1}, m^{\beta+2}}{3, 1+km^{\beta+1}, -3, -1+km^{\beta+1}} / \langle \s_1^{2m} \s_2 \s_1^{-2m} \s_2^{-1} \rangle, \]
	which by \tref{odd-centrals} is the automorphism group of a tight chiral polyhedron of type $\{2m^{\alpha+1}, m^{\beta+2}\}$ (note that $\alpha$ and $\beta$ here are shifted relative to the statement of \tref{odd-centrals}).
	Let $\G_2 = \GP{2r', s'}{-1, 1, 1, -1}$, which is the rotation group of a tight orientably regular polyhedron
	of type $\{2r', s'\}$ by \pref{regular-gp3}. Then inspection of the presentations shows that
	$\G_1 \comix \G_2 = \GP{2,1}{1,1,1,1}$, which has order 2. Then by \pref{mix-size},
	\[ |\G_1 \mix \G_2| = \frac{|\G_1| \cdot |\G_2|}{2} = \frac{(2m^{\alpha+\beta+3})(2r's')}{2} = (2rm)(sm^2). \]
	Then \lref{mixing-tight} implies that $\G_1 \mix \G_2$ is the automorphism group of a tight chiral polyhedron
	of type $\{2rm, sm^2\}$. \\
	
	\par\noindent \textbf{Case 2}: $s$ is even and $\beta = \alpha-1$. We set
	\[ \G_1 = \GP{2m^{\alpha+1}, m^{\beta+2}}{3+k(m+1)m^{\beta+1}, 1+km^{\beta+1}, -3+k(m+1)m^{\beta+1}, -1+km^{\beta+1}} / \langle \s_1^{2m} \s_2 \s_1^{-2m} \s_2^{-1} \rangle, \]
	which \tref{odd-centrals} says is the automorphism group of a tight chiral polyhedron of type
	$\{2m^{\alpha+1}, m^{\beta+2}\}$. As in case 1, we choose $\G_2 = \GP{2r', s'}{-1, 1, -1, 1}$, and the rest
	of the argument works without modification. \\
	
	\par\noindent \textbf{Case 3}: $s$ is even and $\beta \leq \alpha-2$. Set $t = m^{\beta+1} s'/2$ and $u = 2m^{\alpha-1}r'$.
	Then since $u$ is even and $\alpha-1 \geq \beta+1$, the first case implies that there is a tight chiral polyhedron of type
	$\{2tm, um^2\}$. Taking the dual, we get there is a tight chiral polyhedron of type $\{um^2, 2tm\}$.
	Since $um^2 = 2m^{\alpha+1}r' = 2rm$ and $2tm = m^{\beta+2}s' = sm^2$, we get a polyhedron of type $\{2rm, sm^2\}$. \\
	
	\par\noindent \textbf{Case 4}: $s$ is odd and $r$ divides $sm$. Then $\beta+1 \geq \alpha$ and $r'$ divides $s'$. If $\beta+1 > \alpha$
	then we choose $\G_1$ as in case 1; if $\beta+1 = \alpha$ then we choose $\G_1$ as in case 2. In any case,
	$\G_1$ is the automorphism group of a tight chiral polyhedron of type $\{2m^{\alpha+1}, m^{\beta+2}\}$.
	For $\G_2$ we choose $\GP{2r', s'}{3, 1, -3, -1}$; since $r'$ divides $s'$, \pref{regular-gp1} says
	that $\G_2$ is the automorphism group of a tight orientably regular polyhedron of type $\{2r', s'\}$.
	Then $\G_1 \comix \G_2 = \GP{2,1}{1,1,1,1}$, which has order 2. We finish the argument as in case 1,
	showing that $\G_1 \mix \G_2$ gives a tight chiral polyhedron of the desired type.
	\end{proof}
	
	\begin{theorem}
	\label{thm:even-chirals}
	For any positive integers $r$ and $s$, there is a tight chiral polyhedron of type $\{8r, 32s\}$.
	\end{theorem}

	\begin{proof}
	Let us write $r = 2^{\alpha} r'$ and $s = 2^{\beta} s'$, with $r'$ and $s'$ both odd.
	So our goal is to find a tight chiral polyhedron of type $\{2^{\alpha+3}r', 2^{\beta+5}s'\}$.
	We consider 3 cases. \\

	\par\noindent \textbf{Case 1}: $\alpha + 3 < \beta + 5$. We set
	\[ \G_1 = \GP{2^{\alpha+3}, 2^{\beta+5}}{3, 1+2^{\beta-2}, -3, -1+2^{\beta-2}}. \]
	Then by \tref{even-centrals}, $\G_1$ is the automorphism group of a tight chiral polyhedron
	of type $\{2^{\alpha+3}, 2^{\beta+5}\}$. We pick $\G_2 = \GP{2r', 2s'}{-1, 1, 1, -1}$, which by
	\pref{regular-gp3} is the automorphism group of a tight orientably regular polyhedron of type
	$\{2r', 2s'\}$. Then $\G_1 \comix \G_2 = \GP{2, 2}{1, 1, 1, 1}$, which has order 4.
	Then \pref{mix-size} shows that $|\G_1 \mix \G_2| = (2^{\alpha+3}r')(2^{\beta+5} s') = (8r)(32s)$,
	and the result follows from \lref{mixing-tight}. \\
	
	\par\noindent \textbf{Case 2}: $\alpha+3 > \beta+5$. In this case, we pick $\G_1$ to be the dual of what
	we chose for $\G_1$ in case 1. The rest of the argument from case 1 carries through
	without modification. \\
	
	\par\noindent \textbf{Case 3}: $\alpha+3 = \beta+5$. Now we choose $\G_1$ to be the mix of
	\[ \GP{2^{\alpha+2}, 2^{\alpha+3}}{-1 + 2^{\alpha+1}, -3 + 2^{\alpha+1}, 1 + 2^{\alpha+1}, 3 + 2^{\alpha+1}}, \]
	with
	\[ \GP{2^{\alpha+3}, 8}{-1 + 2^{\alpha+1}, -3, 1 + 2^{\alpha+1}, 3}, \]
	as in \tref{even-centrals2}. We again choose $\G_2 = \GP{2r', 2s'}{-1, 1, 1, -1}$.
	If we can show that $\G_1 \comix \G_2$ has order 4, then the result will follow, as in the previous
	cases. In fact, it suffices to show that $|\G_1 \comix \G_2| \geq 4$; that will show (using
	\pref{mix-size}) that $|\G_1 \mix \G_2| \leq (8r)(32s)$. Since we know that $\G_1 \mix \G_2$
	is the rotation group of a polyhedron of type $\{8r, 32s\}$, it follows that
	$|\G_1 \mix \G_2| \geq (8r)(32s)$, and the result will then follow from \lref{mixing-tight}.
	Since $\G_1$ covers $\G_3 := \GP{2^{\alpha+3}, 8}{-1 + 2^{\alpha+1}, -3, 1 + 2^{\alpha+1}, 3}$,
	it follows that $\G_1 \comix \G_2$ covers $\G_3 \comix \G_2$, which is $\GP{2, 2}{1, 1, 1, 1}$
	of order 4. The result now follows.
	\end{proof}

	We summarize with the following theorem:
	
	\begin{theorem}
	\label{thm:existence}
	There is a tight chiral polyhedron of type $\{p, q\}$ and one of type $\{q, p\}$ under any of the following conditions:
	\begin{enumerate}
	\item $q$ is odd, $p$ is an even divisor of $2q$, and there is an odd prime $m$ such that
	$m$ divides $p$ and $m^2$ divides $q$.
	\item $p$ and $q$ are both even, and there is an odd prime $m$ such that
	$m$ divides $p$ and $m^2$ divides $q$.
	\item $p$ is divisible by 8 and $q$ is divisible by 32.
	\end{enumerate}
	\end{theorem}

	Our goal in the following sections will be
	to establish that this is a complete description of the Schl\"afli symbols of tight chiral polyhedra.
	
\section{Atomic Chiral Polyhedra}

	\subsection{Quotients of tight chiral polyhedra}

	We now return to first principles to
	study the structure of tight chiral polyhedra in general. The following simple result
	underlies all of our structure theory.

	\begin{proposition}
	\label{prop:normal-props}
	Suppose $\calP$ is a tight chiral polyhedron of type $\{p, q\}$, with $\G^+(\calP) = \langle \s_1, \s_2 \rangle$.
	If $q \geq p$, then there is an integer $q'$ dividing $q$ and with $2 \leq q' < p$ such that $\langle \s_2^{q'} \rangle$
	is normal in $\G^+(\calP)$.
	\end{proposition}
	
	\begin{proof}
	Let $H = \langle \s_2 \rangle$. Since $\calP$ is tight, there are $p$ right cosets of $H$;
	namely, $H, H \s_1, \ldots, H \s_1^{p-1}$.
	Considering the action of $\langle \s_2 \rangle$ on these cosets, we see that the stabilizer of $H \s_1$ is 
	$\langle \s_2^{q'} \rangle$ for some $q'$ dividing $q$. Since $\s_2^{q'}$ fixes $H \s_1$, that means that
	$\s_1 \s_2^{q'} \s_1^{-1} \in H$, and it follows that $\langle \s_2^{q'} \rangle$ is normal.
	
	Now, the size of the orbit of $H \s_1$ is equal to $q'$. Since $\s_2$
	fixes the coset $H$, the orbit of $H \s_1$ has size at most $p-1$, and so $q' < p$. 
	Next, suppose $q' = 1$. Since $\s_1 \s_2 = \s_2^{-1} \s_1^{-1}$, it follows that $(H \s_1) \s_2 = H \s_2^{-1} \s_1^{-1}
	= H \s_1^{-1}$, and so $\s_1 = \s_1^{-1}$. That forces $p = 2$, but for every $q$, 
	there is only a single polyhedron of type $\{2, q\}$, and that polyhedron is regular. So $q' \geq 2$.
	\end{proof}

	\begin{proposition}
	\label{prop:central-elt}
	Suppose $\calP$ is a tight chiral polyhedron of type $\{p, q\}$, 
	with $\G^+(\calP) = \langle \s_1, \s_2 \rangle$. If $\langle \s_2^{q'} \rangle$ is normal
	in $\G^+(\calP)$, then $\s_1 \s_2^{q'} = \s_2^{aq'} \s_1$ for some $a$ satisfying
	$a^2 \equiv 1$ (mod $q/q'$). Furthermore, $\s_1^2$ commutes with $\s_2^{q'}$,
	and if $p$ is odd, then $\s_2^{q'}$ is central.
	\end{proposition}

	\begin{proof}
	If $\langle \s_2^{q'} \rangle$ is normal in $\G^+(\calP)$, then
	$\s_1 \s_2^{q'} = \s_2^{aq'} \s_1$ for some $a$.
	Then $(\s_1 \s_2) \s_2^{q'} (\s_1 \s_2)^{-1} = \s_2^{aq'}$, and thus $(\s_1 \s_2)^2 \s_2^{q'} (\s_1 \s_2)^{-2}
	= \s_2^{a^2 q'}$. Since $(\s_1 \s_2)^2 = \eps$, it follows that 
	$a^2 q' \equiv q'$ (mod $q$), and so $a^2 \equiv 1$ (mod $q/q'$). It is clear then that $\s_1^2$
	commutes with $\s_2^{q'}$. If $p$ is odd, this implies that $\s_1$ commutes with $\s_2^{q'}$,
	and thus $\s_2^{q'}$ is central.
	\end{proof}
	
	\begin{corollary}
	\label{cor:tight-quos}
	Every tight chiral polyhedron of type $\{p, q\}$ covers a tight chiral or orientably
	regular polyhedron of type $\{p', q\}$ for some $p' < p$ or of type $\{p, q'\}$ for
	some $q' < q$. Furthermore, every tight chiral polyhedron covers 
	a tight orientably regular polyhedron.
	\end{corollary}
	
	\begin{proof}
	Let $\calP$ be a tight chiral polyhedron of type $\{p, q\}$, and let us assume that $q \geq p$; 
	the proof for $q < p$ is analogous. Then \pref{normal-props} guarantees that there is a normal 
	subgroup $\langle \s_2^{q'} \rangle$, with $q'$ dividing $q$ and $2 \leq q' < p$.
	It follows from \pref{inverse-quo-crit} that $\calP$ covers a tight chiral or orientably
	regular polyhedron of type $\{p, q'\}$. If the quotient is chiral, then we repeat the process
	(with the dual), and eventually we must hit an orientably regular quotient.
	\end{proof}

	Let us call a chiral polyhedron of type $\{p, q\}$ \emph{atomic} if it is tight and it does not cover 
	any tight chiral polyhedra of type $\{p', q\}$ or $\{p, q'\}$ with $p' < p$ or $q' < q$. Every tight
	chiral polyhedron that is not itself atomic must cover an atomic chiral polyhedron. Studying the
	atomic chiral polyhedra will thus give us some insight into the structure of tight chiral polyhedra
	in general. We start with:
	
	\begin{proposition}
	\label{prop:gp-is-reg}
	\begin{enumerate}
	\item If $\calP$ is a tight orientably regular polyhedron of type $\{p, q\}$, then 
	$\G^+(\calP) = \GP{p, q}{i, j, -i, -j}$ for some $i$ and $j$.
	\item If $\calP$ is a tight chiral polyhedron of type $\{p, q\}$, then $\G^+(\calP)$
	is a quotient of $\GP{p, q}{i_1, j_1, i_2, j_2}$, where either $i_1 \not \equiv -i_2$ (mod $p$), or
	$j_1 \not \equiv -j_2$ (mod $q$).
	\end{enumerate}
	\end{proposition}
	
	\begin{proof}
	Part (a) follows from \cite[Theorem 3.3]{tight3}. 
	
	The first half of part (b) follows from the fact
	that, if $\calP$ is tight, then $\G^+(\calP) = \langle \s_1 \rangle \langle \s_2 \rangle$.
	For the second half, if $i_1 \equiv -i_2$ and $j_1 \equiv -j_2$, then $\G^+(\calP)$
	is a quotient of $\GP{p, q}{i_1, j_1, -i_1, -j_1}$. This group is already tight, 
	by \pref{gp-props2}(e), and so it follows that $\G^+(\calP) = \GP{p, q}{i_1, j_1, -i_1, -j_1}$.
	However, this group is invariant under the map that sends each $\s_i$ to $\s_i^{-1}$,
	which means that $\calP$ is orientably regular. So if $\calP$ is chiral, then either
	$i_1 \not \equiv -i_2$ or $j_1 \not \equiv -j_2$.
	\end{proof}
	
	Given a subgroup $H$ of $G$, the \emph{core} of $H$ in $G$ is the largest
	subgroup of $H$ that is normal in $G$. We say that $H$ is \emph{core-free} (in $G$)
	if the core of $H$ is trivial. Note that if $N$ is the core of $H$, then
	$H/N$ is core-free in $G/N$.
	
	\begin{proposition}
	\label{prop:core-free}
	If $\calP$ is an atomic chiral polyhedron with $\G^+(\calP) = \langle \s_1, \s_2 \rangle$,
	then either $\langle \s_1 \rangle$ or $\langle \s_2 \rangle$ is core-free.
	\end{proposition}
	
	\begin{proof}
	Suppose that $\calP$ is an atomic chiral polyhedron of type $\{p, q\}$ and
	that there are proper normal subgroups $\langle \s_1^{p'} \rangle$ and $\langle \s_2^{q'} \rangle$ of
	$\G^+(\calP)$. Then $\calP$ covers a tight polyhedron of type $\{p', q\}$ and a tight polyhedron of type $\{p, q'\}$,
	and since $\calP$ is atomic, both of those polyhedra are regular. Now, since $\calP$ is
	a tight polyhedron of type $\{p, q\}$, its automorphism group $\G^+(\calP)$ must be a quotient
	of $\GP{p, q}{i_1, j_1, i_2, j_2}$ for some $i_1, j_1, i_2$, and $j_2$. 
	Let $\calQ_1$ be the tight orientably regular polyhedron of type $\{p', q\}$ that $\calP$ covers,
	and let $\calQ_2$ be the tight orientably regular polyhedron of type $\{p, q'\}$ that $\calP$ covers.
	Then $\G^+(\calQ_1)$ is a quotient of $\GP{p', q}{i_1, j_1, i_2, j_2}$, and \pref{gp-is-reg}(a)
	implies that $j_1 \equiv -j_2$ (mod $q$). Similarly, $\G^+(\calQ_2)$ is a quotient of
	$\GP{p, q'}{i_1, j_1, i_2, j_2}$, and \pref{gp-is-reg}(a) implies that $i_1 \equiv -i_2$
	(mod $p$). Then \pref{gp-is-reg}(b) says that $\calP$ is regular, contradicting our assumptions.
	\end{proof}

	\begin{corollary}
	\label{cor:s1-core-free}
	If $\calP$ is an atomic chiral polyhedron of type $\{p, q\}$ with $q > p$,
	then $\langle \s_1 \rangle$ is core-free.
	\end{corollary}
	
	\begin{proof}
	By \pref{core-free}, either $\langle \s_1 \rangle$ or $\langle \s_2 \rangle$
	must be core-free, and since $q > p$, \pref{normal-props} says that 
	$\langle \s_2 \rangle$ has a nontrivial core.
	\end{proof}
	
	\subsection{Structure of atomic chiral polyhedra}
	
	In the results that follow, we will usually assume that $q > p$. The atomic chiral
	polyhedra with $q < p$ will then be the duals of what we find.

	\begin{theorem}
	\label{thm:atomic-gp}
	Let $\calP$ be an atomic chiral polyhedron of type $\{p, q\}$, with $q > p$. Suppose that the
	core of $\langle \s_2 \rangle$ is $\langle \s_2^{q'} \rangle$. 
	Then $\G^+(\calP) = \GP{p, q}{i, 1+kq', -i, -1-akq'}$ 
	for some integers $i$, $k$, and $a$, where $a^2 \equiv 1$ (mod $q/q'$) and
	$a \neq 1$. Furthermore, $p$ is even.
	\end{theorem}
	
	\begin{proof}
	Since $\calP$ is atomic, taking the quotient by $\langle \s_2^{q'} \rangle$ yields a tight orientably regular polyhedron
	with (the image of) $\langle \s_2 \rangle$ core-free. As a consequence of \cite{tight3}[Thm 3.3], it follows that
	this quotient has group $\GP{p, q'}{i, j, -i, -j}$ for some $i$ and $j$. Then
	\pref{gp-props2}(a) says that $\langle \s_2^{j-1} \rangle$ is normal, and since $\langle \s_2 \rangle$ is core-free,
	it follows that $j = 1$.
	Therefore, $\G^+(\calP)$ satisfies the relations $\s_2^{-1} \s_1 = \s_1^i \s_2^{1+k_1 q'}$
	and $\s_2 \s_1^{-1} = \s_1^{-i} \s_2^{-1 + k_2 q'}$ for some $k_1$ and $k_2$.
	In other words, $\G^+(\calP)$ is a quotient of $\GP{p, q}{i, 1+k_1 q', -i, -1+k_2 q'}$.
	\pref{gp-props2}(e) tells us that this group is already tight, and so it follows that
	$\G^+(\calP)$ is precisely this group. 
	
	Now, by \pref{central-elt}, the relation $\s_1 \s_2^{q'} = \s_2^{aq'} \s_1$
	holds for some $a$ satisfying $a^2 \equiv 1$ (mod $q/q'$). Thus
	$\s_1 \s_2^{k_1 q'} = \s_2^{ak_1 q'} \s_1$. On the other hand,
	the proof of \pref{gp-props2}(a) implies that $\s_1 \s_2^{k_1 q'} = \s_2^{-k_2 q'} \s_1$.
	Thus $ak_1 q' \equiv -k_2 q'$ (mod $q$), and so $\G^+(\calP) = \GP{p, q}{i, 1+k_1 q', -i, -1 - ak_1 q'}$.
	Then \pref{gp-is-reg}(b) says that we need $1+k_1 q' \not \equiv 1+ak_1 q'$ (mod $q$),
	and so we need $a \not \equiv 1$ (mod $q/q'$).
	Finally, if $p$ is odd, then the fact that $\s_1^2$ commutes with $\s_2^{q'}$ implies
	that $\s_1$ commutes with $\s_2^{q'}$; but that would imply that $a = 1$.
	\end{proof}

	Now that we have a general presentation for the group of an atomic chiral polyhedron, it remains
	to determine the allowable values for all of the parameters. 
	We start by finding restrictions on $p$, $q$, and $q'$.

	\begin{lemma}
	\label{lem:q-over-qp}
	Let $\calP$ be an atomic chiral polyhedron of type $\{p, q\}$ with $q > p$, and let
	the core of $\langle \s_2 \rangle$ be $\langle \s_2^{q'} \rangle$. Then $q/q'$ is a prime power.
	\end{lemma}

	\begin{proof}
	Since $\calP$ is tight, its automorphism group $\G^+(\calP)$ is a quotient of
	$\GP{p, q}{i_1, j_1, i_2, j_2}$ for some choice of $i_1, j_1, i_2,$ and $j_2$.
	Suppose that $q/q'$ is not a prime power, so that it has a nontrivial factorization as $q/q' = bc$, where
	$b$ and $c$ are coprime. Since $q = bcq'$ and $\langle \s_2^{q'} \rangle$ is normal, then so are
	$N_1 := \langle \s_2^{bq'} \rangle$ and $N_2 := \langle \s_2^{cq'} \rangle$. For $i \in \{1, 2\}$,
	let $\calQ_i$ be the polyhedron with $\G^+(\calQ_i) = \G^+(\calP) / N_i$. Since $\calP$
	is atomic, both $\calQ_1$ and $\calQ_2$ are regular. Now, $\G^+(\calQ_1)$ is a quotient
	of $\GP{p, bq'}{i_1, j_1, i_2, j_2}$ and $\G^+(\calQ_2)$ is a quotient of
	$\GP{p, cq'}{i_1, j_1, i_2, j_2}$. From \pref{gp-is-reg}(a), it follows that $i_1 \equiv -i_2$ (mod $p$),
	that $j_1 \equiv -j_2$ (mod $bq'$), and that $j_1 \equiv -j_2$ (mod $cq'$). From these
	last two congruences, it follows that $j_1 \equiv -j_2$ (mod $q$) (since $q = bcq'$).
	But then \pref{gp-is-reg}(b) implies that $\calP$ would be regular, so $q/q'$
	must be a prime power after all.
	\end{proof}

	\begin{lemma}
	\label{lem:special-rel}
	Let $\calP$ be an atomic chiral polyhedron of type $\{p, q\}$ with $q > p$ such that
	the core of $\langle \s_2 \rangle$ is $\langle \s_2^{q'} \rangle$. Then there exist
	integers $i$ and $k$ such that, for every integer $n$, 
	\[ \s_2 \s_1^{2n} \s_2^{-1} = \s_1^{n(i-1)} \s_2^{nkq'}. \]
	\end{lemma}

	\begin{proof}
	By \tref{atomic-gp}, we know that in $\G^+(\calP)$, the relation $\s_2^{-1} \s_1 = \s_1^i \s_2^{1+kq'}$
	holds for some $i$ and $k$. Then:
	\[ \s_2 \s_1^2 \s_2^{-1} = \s_1^{-1} \s_2^{-1} \s_1 \s_2^{-1} = \s_1^{i-1} \s_2^{kq'}. \]
	Now, since $\langle \s_2^{q'} \rangle$ is normal, \pref{central-elt} says that $\s_1^2$ commutes with $\s_2^{kq'}$.
	Furthermore, $i-1$ is even (as a consequence of \cite[Lemma 4.8]{tight3}). Therefore, for each $n$, 
	\[ \s_2 \s_1^{2n} \s_2^{-1} = (\s_1^{i-1} \s_2^{kq'})^n = \s_1^{n(i-1)} \s_2^{nkq'}. \qedhere \]
	\end{proof}
	
	\begin{lemma}
	\label{lem:possible-atomics}
	Let $\calP$ be an atomic chiral polyhedron of type $\{p, q\}$ with $q > p$ such that
	the core of $\langle \s_2 \rangle$ is $\langle \s_2^{q'} \rangle$. Then either
	$p = 2m^{\alpha}$ and $q = m^{\alpha + \beta}$ for some positive integers
	$\alpha, \beta$ and for some odd prime $m$, or $p$ and $q$
	are both powers of $2$.
	\end{lemma}			
	
	\begin{proof}
	Taking $n = q/q'$ in \lref{special-rel} yields the relation $\s_2 \s_1^{2q/q'} \s_2^{-1} = \s_1^{q(i-1)/q'}$.
	In particular, $\langle \s_1^{2q/q'} \rangle$ is normal. However, $\langle \s_1 \rangle$
	is core-free in $\G^+(\calP)$, and so this normal subgroup must be trivial. It follows that
	$p$ divides $2q/q'$. Now, \lref{q-over-qp} tells us that $q/q'$ is a prime power; say $q/q' = m^{\beta}$.
	Since $p$ divides $2q/q'$ and $p$ must be even (by \tref{atomic-gp}), it follows that
	$p = 2m^{\alpha}$ for some $\alpha$ satisfying $1 \leq \alpha \leq \beta$.
	
	From \cite[Props. 4.11 and 4.12]{tight3}, $q'$ must divide $p$; in fact, for each odd prime dividing $p$, either
	$q'$ is coprime to that prime, or it contains the full power of that prime. If $m = 2$, then $p$
	is a power of $2$, which forces $q'$ and therefore $q$ to be a power of $2$. Suppose instead
	that $m$ is an odd prime.
	Since $q' < p$, we get that either $q' = 2$ or $q' = m^{\alpha}$. If $q' = 2$, then 
	$\langle \s_2^2 \rangle$ is normal, which implies that $\s_1^{-1} \s_2^2 \s_1
	= \s_2^{2a}$, with $a^2 \equiv 1$ (mod $q/2$) and $a \not \equiv 1$. Since $q/2$ is a power
	of an odd prime, that implies that $a \equiv -1$, and so $\s_1^{-1} \s_2^2 \s_1 = \s_2^{-2}$. Then
	\[ \s_1 \s_2^{-2} = \s_2^2 \s_1 = \s_2 \s_1^{-1} \s_2^{-1}, \]
	and so $\s_2 \s_1^{-1} = \s_1 \s_2^{-1}$. Conjugating by $\s_1$ yields $\s_2^{-1} \s_1 = \s_1^{-1} \s_2$,
	and so $\G^+(\calP) = \GP{p, q}{-1, 1, 1, -1}$, from which it follows from \pref{gp-is-reg}(b) that
	$\calP$ is regular. So $q'$ must be $m^{\beta}$ instead, and thus $q = m^{\alpha+\beta}$.
	\end{proof}

	We are now ready to find restrictions on the parameters $\alpha$ and $\beta$, and to explicitly describe
	the automorphism groups. We start by focusing on the case where $p$ is twice the power of an odd prime.

	\begin{theorem}
	\label{thm:odd-atomic-structure}
	Let $\calP$ be an atomic chiral polyhedron of type $\{2m^{\alpha}, m^{\alpha + \beta}\}$,
	with $m$ an odd prime. Then $\alpha = 1$, and
	$\G^+(\calP) = \GP{2m, m^{\beta+1}}{3, 1+k m^{\beta}, -3, -1+k m^{\beta}}$ for some $k$
	with $1 \leq k \leq m-1$.
	\end{theorem}

	\begin{proof}
	Let $p = 2m^{\alpha}$ and $q=m^{\alpha+\beta}$, and let $\langle \s_2^{q'} \rangle$
	be the core of $\langle \s_2 \rangle$. 
	By \tref{atomic-gp}, $\G^+(\calP) = \GP{p,q}{i, 1+kq', -i, -1-akq'}$ for
	some $k$ and some $a \neq 1$ such that $a^2 \equiv 1$ (mod $q/q'$). 
	Since $q/q'$ is a power of an odd prime, it follows that $a = -1$.
	
	Next, we note that since $\langle \s_2^{q'} \rangle$ is normal, so is $\langle \s_2^{q/m} \rangle$.
	Since $\calP$ is atomic, the quotient of $\G^+(\calP)$ by $\langle \s_2^{q/m} \rangle$ is regular, and so
	$1+kq' \equiv 1 - kq'$ (mod $q/m$), from which it follows that $2kq' \equiv 0$ (mod $q/m$).
	That means that $q/m$ divides $kq'$, and so $mkq' \equiv 0$ (mod $q$).
	Therefore, taking $n = m$ in \lref{special-rel}, we get that
	\[ \s_2 \s_1^{2m} \s_2^{-1} = \s_1^{m(i-1)}. \]
	Then since $\langle \s_1 \rangle$ is core-free, it follows that $p$ divides $2m$, which means that $\alpha = 1$. 
	So $\calP$ is of type $\{2m, m^{\beta+1}\}$. Furthermore, since $\langle \s_2^{q/m} \rangle$ is normal,
	we see that $\calP$ covers a tight orientably regular polyhedron $\calQ$ of type $\{2m, m^{\beta}\}$.
	By \pref{regular-gp1}, it follows that $\G^+(\calQ) = \GP{2m, m^{\beta}}{3, 1, -3, -1}$. So in
	$\G^+(\calQ)$, we have that $\s_2^{-1} \s_1 = \s_1^3 \s_2$, but also $\s_2^{-1} \s_1 = \s_1^i \s_2^{1+kq'}$
	since $\calP$ covers $\calQ$. It follows that $i = 3$.

	It remains to determine $k$. Since $p = 2m$ and $q' < p$, it follows that $q' = m$.
	From before, we have that $mkq' \equiv 0$ (mod $q$), and thus $m^{\beta-1}$ divides $k$.
	Since we can take $k$ modulo $m^{\beta}$, it follows that $k = k' m^{\beta-1}$ for some $k'$
	between $0$ and $m-1$. However, if $k' = 0$ then \pref{gp-is-reg}(b) says that we get a regular polyhedron.
	So $1 \leq k' \leq m-1$, and $kq' = k' m^{\beta}$. The result follows.
	\end{proof}
	
	Note that the groups in \tref{odd-atomic-structure} are precisely those that we had
	in \tref{odd-atomics}.

	It remains to narrow things down for the case where $m = 2$. We start with
	an immediate consequence of \cite[Thm. 4.13]{tight3}.

	\begin{proposition}
	\label{prop:reg-core-free}
	Let $\calP$ be a tight orientably regular polyhedron of type $\{p, q\}$ with
	$\langle \s_1 \rangle$ core-free and with $p$ and $q$ both powers of $2$.
	Then $\G^+(\calP)$ is one of the following:
	\begin{align*}
	& \GP{2, 2^{\alpha}}{-1, 1, 1, -1} \\
	& \GP{4, 2^{\alpha}}{-1, 1 + 2^{\alpha-1}, 1, -1 + 2^{\alpha-1}} \\
	& \GP{2^{\alpha-1}, 2^{\alpha}}{-1, -3, 1, 3} \\
	& \GP{2^{\alpha-1}, 2^{\alpha}}{-1, -3 + 2^{\alpha-1}, 1, 3 + 2^{\alpha-1}}.
	\end{align*}
	\end{proposition}

	Now we can classify the atomic chiral polyhedra of type $\{2^{\alpha}, 2^{\beta}\}$
	with the following two results.
	
	\begin{lemma}
	\label{lem:pow-of-2-atomics}
	Let $\calP$ be an atomic chiral polyhedron of type $\{2^{\alpha}, 2^{\beta}\}$ with
	$\beta > \alpha$. Then the core of $\langle \s_2 \rangle$ in $\G^+(\calP)$ is $\langle \s_2^4 \rangle$,
	and for some $k$ and some $a \neq 1$ satisfying
	$a^2 \equiv 1$ (mod $2^{\beta-2}$), we have
	\[ \G^+(\calP) = \GP{2^{\alpha}, 2^{\beta}}{-1 + 2^{\alpha-1}, 1+4k, 1 + 2^{\alpha-1}, -1 - 4ak}. \]
	\end{lemma}
	
	\begin{proof}
	Let $p = 2^{\alpha}$ and let $q = 2^{\beta}$. Let $q'$ be the integer such that
	$\langle \s_2^{q'} \rangle$ is the core of $\langle \s_2 \rangle$.
	By \tref{atomic-gp}, $\G^+(\calP) = \GP{p, q}{i, 1+kq', -i, -1-akq'}$ for some integers
	$i, k,$ and $a$. Now, since $q$ is a power of 2 and $\langle \s_2^{q'} \rangle$ is a proper normal
	subgroup, it follows that $\langle \s_2^{2^{\beta-1}} \rangle$ is normal in $\G^+(\calP)$.
	Since $\calP$ is atomic, the quotient by this normal subgroup must be regular,
	with group $\GP{2^{\alpha}, 2^{\beta-1}}{i, j, -i, -j}$ for some $i$ and $j$.
	It follows that $\G^+(\calP) = \GP{2^{\alpha}, 2^{\beta}}{i, j+k_1 2^{\beta-1}, -i,
	-j+k_2 2^{\beta-1}}$, with $k_1$ and
	$k_2$ each either 0 or 1. Furthermore, we cannot have $k_1 = k_2$, because
	in that case we get a regular polyhedron (by \pref{gp-is-reg}(b)). So let us suppose that $k_1 = 0$
	and $k_2 = 1$; the other case is analogous. Since $\langle \s_2^{2^{\beta-1}} \rangle$
	is normal and $\s_2^{2^{\beta-1}}$ has order 2, it follows that $\s_2^{2^{\beta-1}}$
	is central. Now, we note that:
	\[ \s_2^{-1} \s_1^{i+1} = \s_1^i \s_2^j \s_1^i = \s_1^{i+1} \s_2^{-1} \s_2^{2^{\beta-1}}, \]
	and thus
	\[ \s_2^{-1} \s_1^{2i+2} \s_2 = \s_1^{2i+2} \s_2^{2^{\beta}} = \s_1^{2i+2}. \]
	Since $\calP$ must have $\langle \s_1 \rangle$ core-free (by \cref{s1-core-free}), 
	it follows that $2^{\alpha}$ divides $2i+2$, so that $i \equiv -1$ (mod $2^{\alpha-1}$).
	Furthermore, if $i = -1$, then the first relation would give us $\s_2^{-1} = \s_2^{-1} \s_2^{2^{\beta-1}}$,
	which would force $\s_2$ to have order $2^{\beta-1}$ instead of $2^{\beta}$. So we must have
	$i = 2^{\alpha-1} - 1$. Then it follows from the dual of \pref{reg-core-free} that $q' = 4$,
	and then \tref{atomic-gp} tells us
	that $a \neq 1$ and $a^2 \equiv 1$ (mod $2^{\beta-2}$). 
	\end{proof}
	
	\begin{theorem}
	\label{thm:even-atomic-structure}
	Let $\calP$ be an atomic chiral polyhedron of type $\{2^{\alpha}, 2^{\beta}\}$ with $\beta > \alpha$.
	Then $\calP$ is either of type $\{8, 2^{\beta}\}$ with $\beta \geq 5$ and 
	\[ \G^+(\calP) = \GP{8, 2^{\beta}}{3, 1 \pm 2^{\beta-2}, -3, -1 \pm 2^{\beta-2}}, \]
	or $\calP$ is of type $\{2^{\alpha}, 2^{\alpha+1}\}$ with $\alpha \geq 4$ and
	\[ \G^+(\calP) = \GP{2^{\alpha}, 2^{\alpha+1}}{-1 + 2^{\alpha-1}, -3 \pm 2^{\alpha-1}, 1 + 2^{\alpha-1}, 3 \pm 2^{\alpha-1}}. \]
	\end{theorem}
	
	\begin{proof}
	Let $p = 2^{\alpha}$ and $q = 2^{\beta}$. Let $\langle \s_2^{q'} \rangle$ be the core of
	$\langle \s_2 \rangle$ in $\G^+(\calP)$. Using \cite{chiral-atlas}, we can verify that
	there are no tight chiral polyhedra of type $\{2^{\alpha}, 2^{\beta}\}$ with $\beta \leq 4$,
	so $\beta \geq 5$.
	
	\lref{pow-of-2-atomics} tells us that $q' = 4$ and that
	$\G^+(\calP) = \GP{p, q}{-1 + \frac{p}{2}, 1+4k, 1 + \frac{p}{2}, -1 - 4ak}$ for
	some $k$ and some $a \neq 1$ satisfying $a^2 \equiv 1$ (mod $q/4$).
	Since $q$ is a power of 2, this implies that $a \equiv \pm 1$ (mod $q/8$).

	First, suppose that $a \equiv -1$ (mod $q/8$). Since $\langle \s_2^4 \rangle$ is normal,
	so is $\langle \s_2^{q/2} \rangle$, and since $\calP$ is atomic, the quotient by this
	normal subgroup is the rotation group of a regular polyhedron. By \pref{gp-is-reg}(a),
	this implies that $1 + 4k \equiv 1 + 4ak$ (mod $q/2$). Then $k \equiv ak$ (mod $q/8$),
	and so $k \equiv -k$ (mod $q/8$). Therefore, $q/16$ divides $k$, and since $q \geq 32$,
	it follows that $k$ is even. Now, substituting $n = 4$ in \lref{special-rel} yields that
	$\s_2 \s_1^8 \s_2^{-1} = \s_1^{4(i-1)} \s_2^{16k} = \s_1^{4(i-1)}$. Since
	$\langle \s_1 \rangle$ is core-free, the subgroup $\langle \s_1^8 \rangle$
	must be trivial, and so $p$ divides $8$. Furthermore, since $q'$ = 4 and
	$q'$ must be less than $p$ (see \pref{normal-props}), it follows that $p = 8$. 
	
	It remains to determine the possible values
	for $4k$ when $p = 8$. We found earlier that $4k$ must be a multiple of $q/4$, and by
	\pref{gp-is-reg}(b), we need $1 + 4k \not \equiv 1 + 4ak$ (mod $q$).
	That rules out $4k = 0$ and $4k = q/2$, and so we must have either
	$4k = q/4$ or $4k = 3q/4$; this gives us the two groups in \tref{even-atomics}.
	
	Now we try the case where $a \equiv 1$ (mod $q/8$). Note that the choice of $a$ only matters modulo
	$q/4$, and since $a \neq 1$ we may as well choose $a = q/8 + 1$. Again, we need $1+4k \not \equiv 1 + 4ak$
	(mod $q$). This implies that $k$ must be odd, and that $-1-4ak \equiv -1-4k+2^{\beta-1}$ (mod $q$).
	Next, using $n = p/2$ in \lref{special-rel} gets us that $\s_2^{2pk} = \eps$,
	and thus $q$ divides $2pk$. Since $k$ is odd and $q$ is a power of two,
	it follows that $q$ divides $2p$. Since
	we must have $q > p$, it follows that $q = 2p$. So $\calP$ has Schl\"afli symbol
	$\{2^{\alpha}, 2^{\alpha+1}\}$.

	To narrow down the possible values of $k$, we note that $\calP$ covers a tight
	orientably regular polyhedron $\calQ$ of type $\{2^{\alpha}, 2^{\alpha}\}$,
	and taking the quotient of $\G^+(\calQ)$ by the core of $\langle \s_1 \rangle$
	gives us one of the groups in \pref{reg-core-free}. The first two choices would
	give us $1+4k \equiv 1$ (mod $2^{\alpha-1}$), and then $k$ would be even.
	So we must have $1+4k \equiv -3$ (mod $2^{\alpha-1}$) instead.

	If $1+4k \equiv -3$ (mod $q$), then $\s_2^{-1} \s_1 = \s_1^{p/2 - 1} \s_2^{-3}$; 
	multiplying both sides on the right by $\s_2$ and using the relation $\s_1 \s_2 = \s_2^{-1} \s_1^{-1}$
	on the left yields
	$\s_2^{-2} \s_1^{-1} = \s_1^{p/2 - 1} \s_2^{-2}$. Inverting both sides yields
	$\s_1 \s_2^2 = \s_2^2 \s_1^{p/2 + 1}$. Then 
	\[ \s_1^2 \s_2^2 = \s_1 \s_2^2 \s_1^{p/2 + 1} = \s_2^2 \s_1^{2(p/2 + 1)} = \s_2^2 \s_1^2. \]
	Similarly, if $1+4k \equiv (q/2)-3$ (mod $q$), then $\s_2 \s_1^{-1} = \s_1^{p/2 + 1} \s_2^3$, from which
	it follows that $\s_2^2 \s_1 = \s_1^{p/2 + 1} \s_2^2$. Then
	\[ \s_2^2 \s_1^2 = \s_1^{p/2+1} \s_2^2 \s_1 = \s_1^2 \s_2^2. \]
	So in either of these cases, $\s_1^2$ and $\s_2^2$ commute. Then we find
	\begin{align*}
	\s_1 \s_2^4 \s_1^{-1} &= (\s_1 \s_2^2 \s_1^{-1})^2 \\
	&= (\s_2^{-1} \s_1^{-1} \s_2 \s_1^{-1})^2 \\
	&= (\s_2^{-2-4k} \s_1^{p/2})^2 \\
	&= \s_2^{-4-8k} \\
	&= \s_2^4.
	\end{align*}
	On the other hand, $\s_1 \s_2^4 \s_1^{-1} = \s_2^{4a} = \s_2^{4+q/2}$.
	This is a contradiction, and so both of these choices for $1+4k$ are invalid.
	The remaining two choices are $1+4k = -3 \pm (q/4)= -3 \pm 2^{\alpha-1}$,
	which give the groups described.
	\end{proof}

	Using the fact that every tight chiral polyhedron covers an atomic chiral polyhedron,
	Theorems~\ref{thm:odd-atomic-structure}~and~\ref{thm:even-atomic-structure} imply the following:
	
	\begin{theorem}
	\label{thm:all-atomics}
	Every tight chiral polyhedron covers an atomic chiral polyhedron of one of the following types:
	\begin{align*}
	\{2m, m^{\alpha}\} &\textrm{ or } \{m^{\alpha}, 2m\} \textrm{ for an odd prime $m$ and $\alpha \geq 2$; or} \\
	\{8, 2^{\beta}\} &\textrm{ or } \{2^{\beta}, 8\} \textrm{ with $\beta \geq 5$; or} \\
	\{2^{\alpha}, 2^{\alpha+1}\} &\textrm{ or } \{2^{\alpha+1}, 2^{\alpha}\} \textrm{ with $\alpha \geq 4$.}
	\end{align*}
	\end{theorem}

\section{Schl\"afli symbols of tight chiral polyhedra}

	We are almost ready to fully characterize the Schl\"afli symbols of tight chiral polyhedra.
	We need one more key result:

	\begin{theorem}
	\label{thm:p-div-2q}
	Let $\calP$ be a tight chiral polyhedron of type $\{p, q\}$, with $q$ odd. Then $p$ is an even divisor of $2q$.
	Furthermore, if $\langle \s_2 \rangle$ is core-free, then $p = 2q$.
	\end{theorem}
	
	\begin{proof}
	First, note that a tight polyhedron $\calP$ of type $\{p, q\}$ with $q$ odd must have $p$ even,
	because $\calP$ has $2pq$ flags, and the number of flags is four times the number of edges.
	
	Now, the claim is true for atomic chiral polyhedra, by \lref{possible-atomics}.
	To prove the general case, we will essentially use induction on the maximal number
	of steps $N$ it takes to pass from $\calP$ to a regular polyhedron via quotients.
	The atomic chiral polyhedra cover the base case $N = 1$. Since $N$ strictly decreases
	as we take quotients, we may assume by inductive hypothesis that any
	quotient of $\calP$ satisfies the claim.
	
	If $q \geq p$, then \pref{normal-props} says that $\calP$ covers a tight chiral or 
	orientably regular polyhedron of type $\{p, q'\}$, with $q'$ odd and $q'$ dividing $q$.
	By inductive hypothesis, $p$ divides $2q'$, which divides $2q$, so the claim is true in this case.
	
	Next, suppose that $p \geq q$. Let $p'$ be the largest divisor of $p$ such that $\langle \s_1^{p'} \rangle$
	is normal in $\G^+(\calP)$ and such that $\langle \ol{\s_2} \rangle$ has nontrivial core
	in $\G^+(\calP) / \langle \s_1^{p'} \rangle$, where $\ol{\s_1}$ and $\ol{\s_2}$ are the
	images of $\s_1$ and $\s_2$. Let $\langle \ol{\s_2}^{q'} \rangle$ be the core of 
	$\langle \ol{\s_2} \rangle$ in $\G^+(\calP) / \langle \s_1^{p'} \rangle$. So
	$\ol{\s_1}^{-1} \ol{\s_2}^{q'} \ol{\s_1} = \ol{\s_2}^{bq'}$ for some $b$, and thus
	for some $a$ the relation $\s_1^{-1} \s_2^{q'} \s_1 = \s_1^{ap'} \s_2^{bq'}$ holds
	in $\G^+(\calP)$. Therefore, if we take the quotient of $\G^+(\calP)$ by
	$\langle \s_1^{ap'} \rangle$, then the image of $\langle \s_2^{q'} \rangle$
	is normal in the quotient. By our choice of $p'$, it follows that
	$a$ is coprime to $p/p'$.

	Now, $\s_1^{p'}$ is central, by (the dual version of) \pref{central-elt}. Then the relation
	$\s_1^{-1} \s_2^{q'} \s_1 = \s_1^{ap'} \s_2^{bq'}$ implies that, for each $n$,
	\begin{equation}
	\label{eq:special-rel}
	\s_1^{-1} \s_2^{nq'} \s_1 = \s_1^{nap'} \s_2^{nbq'}.
	\end{equation}
	Taking $n = q/q'$ yields $\s_1^{ap'q/q'} = \eps$, and thus
	$p$ divides $ap'q/q'$. Therefore, $p/p'$ divides $aq/q'$,
	and since $a$ is coprime to $p/p'$, it follows that $p/p'$ divides $q/q'$.
	
	Now, in $\G^+(\calP) / \langle \s_1^{p'} \rangle$, the subgroup $\langle \ol{\s_2} \rangle$ has
	nontrivial core $\langle \ol{\s_2}^{q'} \rangle$, and taking the quotient by this normal
	subgroup yields the rotation group of a tight chiral or orientably regular polyhedron
	$\calQ$ of type $\{p', q'\}$. Furthermore, $q'$ is odd, and the image of $\langle \ol{\s_2} \rangle$
	is core-free. If $\calQ$ is chiral, then $p' = 2q'$ by inductive hypothesis, whereas if
	$\calQ$ is regular, then $p' = 2q'$ by \cite[Prop. 4.11]{tight3}.
	Combining this with the fact that $p/p'$ divides $q/q'$, we get that $p$ divides $2q$.

	Finally, suppose that $\langle \s_2 \rangle$ is core-free.
	Taking $n = p/p'$ in \eref{special-rel} yields that $\langle \s_2^{pq'/p'} \rangle$
	is normal in $\G^+(\calP)$. Since $\langle \s_2 \rangle$ is core-free,
	it follows that $q$ divides $pq'/p'$, and so $q/q'$ divides $p/p'$.
	Since we also have that $p/p'$ divides $q/q'$, it follows that $q/q' = p/p'$,
	and combining again with the fact that $p' = 2q'$ we get that $p = 2q$.
	\end{proof}
	
	\begin{theorem}
	\label{thm:all-tight-chirals}
	There is a tight chiral polyhedron of type $\{p, q\}$ if and only if one of
	the following is true:
	\begin{enumerate}
	\item[(1)] $q$ is odd, $p$ is an even divisor of $2q$, and there is an odd prime $m$ such that
	$m$ divides $p$ and $m^2$ divides $q$.
	\item[(2)] $p$ is odd, $q$ is an even divisor of $2p$, and there is an odd prime $m$ such that
	$m$ divides $q$ and $m^2$ divides $p$.
	\item[(3)] $p$ and $q$ are both even, and there is an odd prime $m$ such that
	$m$ divides $p$ and $m^2$ divides $q$.
	\item[(4)] $p$ and $q$ are both even, and there is an odd prime $m$ such that
	$m$ divides $q$ and $m^2$ divides $p$.
	\item[(5)] $p$ is divisible by 8 and $q$ is divisible by 32.
	\item[(6)] $q$ is divisible by 8 and $p$ is divisible by 32.
	\end{enumerate}
	\end{theorem}

	\begin{proof}
	\tref{existence} proves that these conditions suffice. To prove necessity, we
	note that \tref{all-atomics} implies that either there is an odd prime $m$
	such that $m$ divides $p$ and $m^2$ divides $q$ (or vice-versa), or that
	$8$ divides $p$ and $32$ divides $q$ (or vice-versa). Combining with
	\tref{p-div-2q} to handle the case where $p$ or $q$ is odd completes the proof.
	\end{proof}
	
\section{Conclusions and future directions}

	In this paper, we have determined the Schl\"afli symbols that occur among tight chiral
	polyhedra. A natural next step would be to obtain a complete classification of the
	automorphism groups of tight chiral polyhedra. The work here with atomic chiral
	polyhedra is already a step in that direction.
	
	Another natural direction for expansion would be the classification of tight
	chiral polytopes of higher dimensions. Work on this problem has already
	begun, including a proof that there are no tight chiral $n$-polytopes
	for $n \geq 6$. Many families of tight chiral $4$-polytopes have been found,
	but so far, no tight chiral $5$-polytopes have been found.

\section{Acknowledgement}

	The author would like to thank Daniel Pellicer for helpful discussions on this topic,
	and the anonymous referees for a number of good suggestions.
	
\bibliographystyle{amsplain}
\bibliography{gabe}

\end{document}